\documentclass[reqno,11pt,final]{amsart}

\usepackage[
paper=a4paper,
headsep=17pt,footskip=15pt,text={155mm,235mm},centering
]{geometry}

\usepackage{amssymb, amsthm, mathabx}
\usepackage[numbers]{natbib}
\usepackage{mathtools}
\usepackage[all]{xy}
\usepackage{float}
\usepackage{enumerate}
\usepackage[colorlinks=true,allcolors=blue,bookmarksopen=true]{hyperref}

\iftrue
  \makeatletter
  \gdef\@settitle{%
    \vspace*{3pt}
    \begin{flushleft}%
      \Large\bfseries
      \strut\@title\strut
    \end{flushleft}%
  }
  \gdef\@setauthors{%
    \begingroup
    \def\thanks{\protect\thanks@warning}%
    \trivlist
    \raggedright
    \large \@topsep31\p@\relax
    \advance\@topsep by -\baselineskip
    \item\relax
    \author@andify\authors
    \def\\{\protect\linebreak}%
    \authors
    \ifx\@empty\contribs
    \else
      ,\penalty-3 \space \@setcontribs
      \@closetoccontribs
    \fi
    \normalfont
    \endtrivlist
    \endgroup
  }
  \gdef\@setaddresses{\par
    \nobreak \begingroup
    \small\raggedright
    \def\author##1{\nobreak\addvspace\smallskipamount}%
    \def\\{\unskip, \ignorespaces}%
    \interlinepenalty\@M
    \def\address##1##2{\begingroup
      \par\addvspace\bigskipamount\noindent
      \@ifnotempty{##1}{(\ignorespaces##1\unskip) }%
      {\ignorespaces##2}\par\endgroup}%
    \def\curraddr##1##2{\begingroup
      \@ifnotempty{##2}{\nobreak\noindent\curraddrname
        \@ifnotempty{##1}{, \ignorespaces##1\unskip}\/:\space
        ##2\par}\endgroup}%
    \def\email##1##2{\begingroup
      \@ifnotempty{##2}{\nobreak\noindent E-mail address%
        \@ifnotempty{##1}{, \ignorespaces##1\unskip}\/:\space
        \ttfamily##2\par}\endgroup}%
    \def\urladdr##1##2{\begingroup
      \def~{\char`\~}%
      \@ifnotempty{##2}{\nobreak\noindent\urladdrname
        \@ifnotempty{##1}{, \ignorespaces##1\unskip}\/:\space
        \ttfamily##2\par}\endgroup}%
    \addresses
    \endgroup
    \global\let\addresses=\@empty
  }
  \gdef\@setabstracta{%
      \ifvoid\abstractbox
    \else
      \skip@19pt \advance\skip@-\lastskip
      \advance\skip@-\baselineskip \vskip\skip@
      \box\abstractbox
      \prevdepth\z@ 
      \vskip-28pt
    \fi
  }
  \renewenvironment{abstract}{%
    \ifx\maketitle\relax
      \ClassWarning{\@classname}{Abstract should precede
        \protect\maketitle\space in AMS document classes; reported}%
    \fi
    \global\setbox\abstractbox=\vtop \bgroup
      \normalfont\small
      \list{}{\labelwidth\z@
        \leftmargin0pc \rightmargin\leftmargin
        \listparindent\normalparindent \itemindent\z@
        \parsep\z@ \@plus\p@
        
      }%
      \item[\hskip\labelsep\bfseries\abstractname.]%
  }{%
    \endlist\egroup
    \ifx\@setabstract\relax \@setabstracta \fi
  }

  \gdef\ps@headings{\ps@empty
    \def\@evenhead{%
      \setTrue{runhead}%
      \normalfont\scriptsize
      \rlap{\thepage}\hfill
      \def\thanks{\protect\thanks@warning}%
      \leftmark{}{}}%
    \def\@oddhead{%
      \setTrue{runhead}%
      \normalfont\scriptsize
      \def\thanks{\protect\thanks@warning}%
      \rightmark{}{}\hfill \llap{\thepage}}%
    \let\@mkboth\markboth
  }\ps@headings

  \gdef\section{\@startsection{section}{1}%
    \z@{-1.4\linespacing\@plus-.5\linespacing}{.8\linespacing}%
    {\normalfont\bfseries\large}}
  \gdef\subsection{\@startsection{subsection}{2}%
    \z@{-.8\linespacing\@plus-.3\linespacing}{.5\linespacing\@plus.2\linespacing}%
    {\normalfont\bfseries}}
  \gdef\subsubsection{\@startsection{subsubsection}{3}%
    \z@{.7\linespacing\@plus.2\linespacing}{-1.5ex}%
    {\normalfont\bfseries}}
  \gdef\@secnumfont{\bfseries}

  \renewcommand\contentsnamefont{\bfseries}
  \gdef\@starttoc#1#2{\begingroup
    \setTrue{#1}%
    \par\removelastskip\vskip\z@skip
    \@startsection{}\@M\z@{\linespacing\@plus\linespacing}%
      {.5\linespacing}{
        \contentsnamefont}{#2}%
    \ifx\contentsname#2%
    \else \addcontentsline{toc}{section}{#2}\fi
    \makeatletter
    \@input{\jobname.#1}%
    \if@filesw
      \@xp\newwrite\csname tf@#1\endcsname
      \immediate\@xp\openout\csname tf@#1\endcsname \jobname.#1\relax
    \fi
    \global\@nobreakfalse \endgroup
    \addvspace{32\p@\@plus14\p@}%
    \let\tableofcontents\relax
  }
  \gdef\contentsname{Contents}
  \gdef\l@section{\@tocline{2}{.5ex}{0mm}{5pc}{}}
  \gdef\l@subsection{\@tocline{2}{0pt}{2em}{5pc}{}}
  \makeatother
\fi 

\theoremstyle{plain}
\newtheorem{thm}{Theorem}[section]
\newtheorem{lem}[thm]{Lemma}

\newtheorem{prop}[thm]{Proposition}
\newtheorem{claim}[thm]{Claim}
\newtheorem{conj}[thm]{Conjecture}

\theoremstyle{definition}
\newtheorem{exl}[thm]{Example}
\newtheorem{defn}[thm]{Definition}
\newtheorem{remark}[thm]{Remark}
\newtheorem{quest}[thm]{Question}

\DeclareMathOperator{\sr}{\Sigma_2(R)}

\DeclareMathOperator{\Z}{\mathbb{Z}}
\DeclareMathOperator{\Q}{\mathbb{Q}}

\DeclareMathOperator{\F}{\mathbb{F}}
\DeclareMathOperator{\ssm}{\smallsetminus}
\DeclareMathOperator{\Sig}{\Sigma}
\DeclareMathOperator{\GL}{GL}
\DeclareMathOperator{\Arf}{Arf}
\DeclareMathOperator{\Imm}{Im}
\DeclareMathOperator{\Hom}{Hom}
\DeclareMathOperator{\Ext}{Ext}
\DeclareMathOperator{\coker}{coker}
\DeclareMathOperator{\grr}{g-rk}
\DeclareMathOperator{\Id}{Id}
\DeclareMathOperator{\Aut}{Aut}
\DeclareMathOperator{\Bl}{Bl}

\newcommand{\wt}{\widetilde}

\begin{document}

\title{Two-solvable and two-bipolar knots with large four-genera}

\author{Jae Choon Cha}
\address{
  Department of Mathematics\\
  POSTECH\\
  Pohang Gyeongbuk 37673\\
  Republic of Korea\linebreak
  School of Mathematics\\
  Korea Institute for Advanced Study \\
  Seoul 02455\\
  Republic of Korea
}
\email{jccha@postech.ac.kr}

\author{Allison N.~Miller}
\address{
  Department of Mathematics\\
  Rice University\\
  Houston, TX, USA}
\email{allison.miller@rice.edu}

\author{Mark Powell}
\address{
  Department of Mathematical Sciences\\
  Durham University\\
  United Kingdom}
\email{mark.a.powell@durham.ac.uk}

\def\subjclassname{\textup{2010} Mathematics Subject Classification}
\expandafter\let\csname subjclassname@1991\endcsname=\subjclassname
\expandafter\let\csname subjclassname@2000\endcsname=\subjclassname

\begin{abstract}
For every integer $g$, we construct a $2$-solvable and $2$-bipolar knot whose
topological $4$-genus is greater than~$g$. Note that $2$-solvable knots are in
particular algebraically slice and have vanishing Casson-Gordon obstructions.
Similarly all known smooth 4-genus bounds from gauge theory and Floer homology
vanish for $2$-bipolar knots.  Moreover, our knots bound smoothly embedded height
four gropes in $D^4$, an a priori stronger condition than being $2$-solvable. We use
new lower bounds for the $4$-genus arising from $L^{(2)}$-signature defects
associated to meta-metabelian representations of the fundamental
group.
\end{abstract}

\maketitle

\section{Introduction}

A knot $K$ in $S^3$ is \emph{slice} if there exists a locally flat proper embedding
$D^2 \hookrightarrow D^4$ such that the boundary of $D^2$ is the knot $K$. This idea of `4-dimensional triviality' can be
generalized in a number of ways, perhaps most easily by approximating a disc by
a small genus surface. The \emph{$4$-genus} $g_4(K)$ of a knot $K$ in $S^3$ is
the minimal possible genus $g(\Sigma)$ of an orientable surface~$\Sigma$ with a locally flat proper embedding $\Sigma
\hookrightarrow D^4$  in the 4-ball, where
$\Sigma$ has a single boundary component whose image coincides with~$K$. From this point of view, a knot is approximately
slice if it has small 4-genus.  However, this perspective does not give successively closer approximations to sliceness; there also exist many knots of $4$-genus one, such as the trefoil, which intuitively seem far from slice.

An alternative approach is to approximate the slice disc exterior $X_D := D^4 \ssm
\nu(D^2)$, a compact $4$-manifold with the three key properties that
(i)~$\partial X_D=M_K$, the 0-surgery of $S^3$ along~$K$; (ii)~the inclusion
induces an isomorphism $H_1(M_K) \cong H_1(X_D)$; and (iii)~$H_2(X_D)=0$. We
therefore think of a compact 4-manifold $W$ with $\partial W= M_K$ such that  $i_* \colon H_1(M_K)\to
H_1(W)$ is an isomorphism and some condition on $H_2(W)$ is satisfied
as an approximation to a slice disc exterior. One might ask that $H_2(X_D)$ is of small rank, but a little thought shows that this essentially recovers the
4-genus condition, besides again not yielding arbitrarily refined
approximations.

 In~\cite{Cochran-Orr-Teichner:1999-1}, Cochran, Orr, and
Teichner introduced a new perspective, motivated by surgery theory, in which one
allows $H_2(W)$ to be arbitrarily large but requires that it is generated by
almost disjointly embedded surfaces with a condition on the image of their fundamental groups in~$\pi_1(W)$. See
Section~\ref{section:solvable} for the precise definition.  In fact, they give
an infinite family of increasingly strict conditions, indexed by $h \in \frac12
\mathbb{N}$: a knot is said to be \emph{$h$-solvable} if its 0-surgery bounds a
slice disc exterior approximation satisfying the $h$th such condition.   It is an open question whether any knot which
is $h$-solvable for all $h$ must be slice, and in general knots which are
$h$-solvable for large $h$ are hard to distinguish from slice knots.

The idea of solvability is closely related to the more geometric notion of bounding a \emph{grope} of large height.  A grope of height 1 is defined to be an orientable surface of arbitrary genus and a single boundary component, and a grope of height $n$ is obtained by attaching boundaries of gropes of height $n-1$ to an orientable surface along standard basis curves.  We refer to
\cite{Freedman-Quinn:1990-1,Cochran-Orr-Teichner:1999-1}, or our
Section~\ref{section:height-four-gropes} for the precise definition.
A grope of larger height is a better approximation to a disc.
Gropes are ingredients of fundamental importance for the topological disc embedding technology of Freedman and Quinn~\cite{Freedman:1984-1,Freedman-Quinn:1990-1} on 4-manifolds, and also in the work of Cochran, Orr
and Teichner~\cite{Cochran-Orr-Teichner:1999-1} discussed above, where it was shown that if a knot $K$ bounds an embedded framed grope of height $h$ in $D^4$ then $K$ is $(h-2)$-solvable. The converse remains an open question.

It is natural to ask whether the 4-genus and grope/$n$-solvability approximations to sliceness have any relationship.

\begin{quest}[{\cite[Remark~5.6]{Cha:2006-1}}]
  \label{quest:main}
  For a fixed $h$, do there exist $h$-solvable knots, i.e.\ knots which are
  close to slice in the sense of \cite{Cochran-Orr-Teichner:1999-1}, which have
  arbitrarily large 4-genera, and hence are far from slice in the first sense?
\end{quest}

This question seems to be difficult, one reason for which is that existing methods for extracting lower bounds for the
topological 4-genus are not effective for $h$-solvable knots with $h\ge 2$. The simplest lower bounds are the Tristram-Levine signature function and Taylor's bound~\cite{Taylor:1979}, the best possible bound for the 4-genus coming from the Seifert form.  For algebraically slice knots these lower bounds vanish.
In~\cite{Gilmer:1982-1}, Gilmer showed that there are algebraically slice knots with arbitrarily large 4-genus using Casson-Gordon
signatures~\cite{Casson-Gordon:1978-1, Casson-Gordon:1986-1}.
In~\cite{Cha:2006-1}, Cha showed that there exist knots with arbitrarily large
4-genus which are algebraically slice and have vanishing Casson-Gordon
signatures, using Cheeger-Gromov Von Neumann $L^{(2)}$ $\rho$-invariants
corresponding to metabelian fundamental group representations. The above
abelian and metabelian lower bounds can be used to give affirmative answers to
Question~\ref{quest:main} for the initial cases $h=0$,~$1$,  but these lower bounds vanish for $h$-solvable knots with $h\ge 2$.  Extending Cha's $\rho$-invariant approach beyond the metabelian level to give further lower bounds for the 4-genus was left open, essentially because of difficulties arising from non-commutative algebra.

In this paper, we present a new method that avoids the non-commutative algebra problem.  It enables us to go one step further than Gilmer and Cha, by combining a Casson-Gordon type approach and $L^{(2)}$-signatures associated with representations to 3-solvable groups i.e.\ solvable groups with length 3 derived series.  Here is our main result.


\begin{thm}
  \label{thm:mainthm-intro}
  For each $g \in \mathbb{N}$, there exists a  2-solvable knot $K$
  with $g_4(K) >  g$. Moreover, $K$ bounds an embedded framed grope of height 4 in
  $D^4$.
\end{thm}



Moreover, the knots of Theorem~\ref{thm:mainthm-intro} are  \emph{2-bipolar} in the sense of Cochran, Harvey
and Horn~\cite{Cochran-Harvey-Horn:2012-1}.  We give the definition in
Section~\ref{section:solvable}, noting for now that the notion of bipolarity is
an approximation to being smoothly slice, which combines the idea of Donaldson's
diagonalization theorem with fundamental group information related to gropes
and derived series.  Also, for a 2-bipolar knot, the invariants $\tau$,
$\Upsilon$, $\varepsilon$, $\nu^+$ from Heegaard-Floer homology, as well as the
$d$-invariants of $p/q$ surgery, all cannot prove that the knot is not smoothly
slice, and consequently cannot bound the smooth
$4$-genus~\cite{Cochran-Harvey-Horn:2012-1}. This also holds for gauge theoretic
obstructions such as those arising from Donaldson's theorem and the $10/8$
theorem.

Theorem~\ref{thm:mainthm-intro} answers the $h=2$ case of
Question~\ref{quest:main}, and prompts us to conjecture that the answer is `yes' in general. In fact, we make a bolder conjecture.

\begin{conj}\label{conj:stablegenus}
  Let $K$ be an $h$-solvable knot which is not torsion in $\mathcal{C}$. Then
  $\{\#^n K\}$ is a collection of $h$-solvable knots containing knots with
  arbitrarily large 4-genera.
\end{conj}

A knot which did not satisfy the second sentence would be an example of a
non-torsion knot with stable 4-genus zero i.e.\ $\lim_{n \to \infty} g_4(nK)/n =
0$, and it is unknown whether any such knots exist~\cite{Livingston:2010}. Thus
a counterexample to this conjecture would also be very interesting.

One might also wonder whether there exist highly bipolar knots with large smooth
4-genus, especially with the additional requirement that they be topologically
slice.  The following question seems to be unknown even in the case $h=0$.

\begin{quest}
  Do there exist topologically slice $h$-bipolar knots with large smooth
  4-genus?
\end{quest}

As above, there are many reasonable candidate knots with which one might hope to answer `yes.'
 The main result of~\cite{Cha-Kim:2017} gave many examples of
topologically slice, $h$-bipolar knots $K$ which are of infinite order, even
modulo the subgroup of $(h+1)$-bipolar knots, and a smooth/bipolar analogue of
Conjecture~\ref{conj:stablegenus} suggests we should expect $\#^n K$ to have
arbitrarily large smooth 4-genus as $n \to \infty$.

\subsubsection*{Summary of the construction and proof}

In order to construct 2-bipolar knots bounding height four gropes, we
take connected sums of sufficiently many  copies of the
seed ribbon knot $R:=11_{n74}$, and perform satellite operations on a collection
of judiciously chosen infection curves $\{\alpha_i^+,\alpha_i^-\}$, with
$\alpha_i^{\pm}$ lying in the second derived subgroup $\pi_1(S^3 \ssm R)^{(2)}$
of the knot group of the $i$th copy of $R$. Our choice of $(R, \alpha^+, \alpha^-)$ is depicted on the right side of Figure~\ref{fig:11n742}.
We use knots $\{J_i^+,J_i^-\}$ with Arf invariant zero for the companions of the satellite operations,  chosen so that
the $\{J_i^+\}$ have increasingly large negative Tristram-Levine signature
functions and the $\{J_i^-\}$ have increasingly large positive signature
functions.

Let $K$ be the result of these satellite operations. In
Proposition~\ref{prop:infection2solvable}, we show that $K$ is 2-solvable; in
Proposition~\ref{prop:buildingbipolar}, we  show that $K$ is 2-bipolar; and in
Proposition~\ref{prop:gropebounding}, we show that $K$ bounds a grope of
height 4 in~$D^4$. Writing $K_i$ for the knot resulting from the satellite
construction on $(R,\alpha_i^{\pm}, J_i^{\pm})$, we have $K = \#_{i=1}^N K_i$.
Let $M_{K_i}$ be the zero-surgery manifold of $K_i$ and write $Y:=
\bigsqcup_{i=1}^N M_{K_i}$.

The main idea of our proof is as follows.  If there were a surface $\Sigma$ of genus $g$ embedded in $D^4$ with boundary $K$, then there would be an associated 4-manifold $Z$ with boundary $Y$ and a quotient $\Gamma$ of $\pi_1(Z)$ such that the $L^{(2)}$ $\rho$-invariant \[\rho^{(2)}(Y,\Gamma) := \rho^{(2)}(Y,\phi \colon \pi_1(Y) \to \pi_1(Z) \to \Gamma)\] would be bounded above by a constant depending only on $g$ and  the base knot $R$. However, by choosing the infection knots $\{J_i^{\pm}\}$ to have suitably large Tristram-Levine signature functions, $L^{(2)}$-induction will imply that $\rho^{(2)}(Y,\Gamma)$ must be very large so long some curve $\alpha_i^{\pm}$ represents an element of $\pi_1(Y)$  mapping nontrivially to $\Gamma$. The key difficulty is to show that this must always be the case, recalling that $\Gamma$ depends on the hypothesized surface $\Sigma$.

%
%
%

In Example~\ref{subsection:example2solvable} we
present a slightly simpler construction of a family of $2$-solvable knots with
arbitrary 4-genera, starting with connected sums of the ribbon knot
$8_8$ and performing a single satellite construction on each copy of~$8_8$ as indicated in Figure~\ref{fig:88}.

\subsubsection*{Coefficient systems: comparison with earlier methods}

To show the nontriviality of some $\alpha_i^{\pm}$ in $\Gamma$, we use twisted
homology over a metabelian representation to define the coefficient system.
Although the representation is non-abelian, we use the ideas of Casson and
Gordon~\cite{Casson-Gordon:1986-1} to define finitely generated twisted homology
modules over a \emph{commutative} principal ideal domain.  The commutativity enables us to consider the ``size'' of the twisted homology modules in terms of the minimal number of generators, generalizing the abelian representation case in
e.g.~\cite{Cha:2006-1}.  Supposing that the 4-genus is small compared to the
size of the twisted first homology, we show that there is a \emph{meta-metabelian}
quotient $\Gamma$ of $\pi_1(Z)$, i.e. a quotient whose third derived subgroup vanishes,  in which one of the $\alpha_i^{\pm}$ is nontrivial in order to eventually obtain a contradiction.  In previous approaches to slice obstructions using  $L^{(2)}$-signature defects corresponding to representations to groups with nontrivial $n$th derived subgroups for $n \geq 2$, the homology modules associated to non-abelian representations were over non-commutative rings, for which it is still unknown how to implement an analogous generating rank argument.

In our method, it is also crucial to use $L^{(2)}$-signatures
over amenable groups that are \emph{not torsion-free}, which were developed
in~\cite{Cha-Orr:2009-01, Cha:2014-1} and deployed in a similar context in~\cite{MilPow17}.

\subsubsection*{The smooth slice genus}

We remark that concordance obstructions predicated on a smooth embedding cannot
be used to draw conclusions about locally flat surfaces, and hence cannot be
used to prove our result. On the other hand our knots have arbitrarily large
smooth 4-genus, since a smooth embedding of a surface in $D^4$ is in particular
a locally flat embedding.

If we were interested in the smooth 4-genus version
of Theorem~\ref{thm:mainthm-intro}, currently known techniques using Heegaard
Floer homology or gauge theory would not apply.  It is unknown whether the
Rasmussen $s$-invariant, which does provide a lower bound on the smooth 4-genus
of a knot, must vanish for 2-bipolar knots.  Our knots are even the first
examples in the literature of 1-bipolar knots with large 4-genus, though for that result one could use a simpler Casson-Gordon signature argument analogous to~\cite{Gilmer:1982-1}.

\subsubsection*{Horn's results}

The fact that $g_4(K)$ is large implies that the base surface of any embedded
grope in $D^4$ with boundary $K$ must have large genus. The main theorem of
Horn~\cite{Horn:2011} gives examples, for each $g$ and each $n$, of knots
bounding height $n$ gropes such that the base surface of any height $n$ grope
must have genus at least $g$. However, Horn's example knots are not known to
have large 4-genera: he was only able to provide lower bounds on the genera of
surfaces that extend to an embedding of a height $n$ grope.

\subsubsection*{Organization of the paper}

The next four sections are concerned with background theory.
Section~\ref{section:solvable} recalls the definitions of the derived series of
a group, a useful variation called the local derived series, and what it means
for a knot to be $h$-solvable or $h$-bipolar.  We also explain here how to
construct $h$-solvable  and $h$-bipolar knots using the satellite construction.
Section~\ref{section:disconnected} introduces some conventions for dealing with
disconnected manifolds, in particular as relates to representations of their
fundamental groupoids and associated twisted homology groups.
Section~\ref{section:L2-invariants} recalls the Cheeger-Gromov von Neumann
$L^{(2)}$ $\rho$-invariant $\rho^{(2)}(Y,\phi)$ of a closed 3-manifold $Y$
together with a homomorphism of its fundamental group $\pi_1(Y) \to \Gamma$ to a
group $\Gamma$, and gives the facts about this invariant that we will need.
Section~\ref{section:metabelian-homology} describes homology twisted with
metabelian representations.  In particular we consider coefficient systems
inspired by Casson-Gordon invariants~\cite{Casson-Gordon:1986-1}.

Section~\ref{section:statement} begins the proof of
Theorem~\ref{thm:mainthm-intro}, by precisely stating the criteria that will
imply certain knots have large topological 4-genus, giving a brief outline of
the proof, and providing examples meeting those criteria.
Section~\ref{section:controlling-homology-groups} proves some technical lemmas
that are vital in arranging that the representation used for our
$\rho$-invariant computation is suitably nontrivial. For this, we control the
size of the homology groups of certain covering spaces. In
Section~\ref{section:a-standard-cobordism}  we review a standard cobordism used
in the proof of Theorem~\ref{thm:mainthm-intro}, and carefully  investigate the
way metabelian representations extend over this cobordism.
Section~\ref{section:proof-of-main-theorem} proves the main theorem by bounding
the $\rho$-invariant in two different ways as described above.
Section~\ref{section:height-four-gropes} proves that our knots bound height four
embedded gropes.

\subsubsection*{Acknowledgements}

The first and third authors thank the Max Planck Institute for Mathematics in Bonn, where they were visiting when part of the work on this paper occurred. The second author thanks Shelly Harvey for stimulating conversations.   The first author was partly supported by NRF grant 2019R1A3B2067839.  Finally, we thank the anonymous referee for a careful reading and useful suggestions which improved the paper.

\section{The solvable and bipolar filtrations}\label{section:solvable}

In this section we recall the definitions of the solvable and bipolar
filtrations, and how to construct highly solvable or bipolar knots.  We will
also need, later in the article, not just the standard derived series of a group
but also the local derived series~\cite{Cochran-Harvey:2004-1,
Cochran-Harvey:2007-01, Cha:2014-1}.

\begin{defn}\label{defn:nderived}
Let $G$ be a group. The \emph{$h$th derived subgroup}  $G^{(h)}$ of $G$  is
defined recursively via $G^{(0)}:=G$ and $G^{(h)}= [G^{(h-1)}, G^{(h-1)}]$ for
$h\geq1$. Moreover, for any sequence  $\mathcal{S}= (S_i)_{i \in \mathbb{N}}$ of
abelian groups, define the \emph{$h$th $\mathcal{S}$-local derived subgroup} of
$G$ recursively by $G_\mathcal{S}^{(0)}:=0$ and, for $h \geq 1$,
\[
  G_{\mathcal{S}}^{(h)}:= \ker\Bigl\{ G^{(h-1)}_\mathcal{S}
  \to G^{(h-1)}_\mathcal{S} / [ G^{(h-1)}_\mathcal{S} ,  G^{(h-1)}_\mathcal{S} ]
  \to \bigl( G^{(h-1)}_\mathcal{S} /
  [ G^{(h-1)}_\mathcal{S} ,  G^{(h-1)}_\mathcal{S} ] \bigr)
  \otimes_{\Z} S_{h} \Bigr\}.
\]
\end{defn}
We remark that a group $G$ is called \emph{metabelian} if $G^{(1)} \neq 0$ but $G^{(2)}=0$ and analogously \emph{meta-metabelian} if $G^{(2)} \neq 0$ but $G^{(3)}=0$. This explains some language from the introduction.

For any sequence $\mathcal{S}$ and any $h \in \mathbb{N}$ we have
that $G^{(h)} \subseteq G^{(h)}_\mathcal{S}$. Note that since for fixed $h \in \mathbb{N}$ the subgroup $G^{(h)}_\mathcal{S}$ only depends on the first $h$ terms of $\mathcal{S}$,
we will often take $\mathcal{S}=(S_1, \dots, S_h)$ to be a partial sequence.
We will be particularly interested
in $\mathcal{S}= (\Q, \Z_{p}, \Q)$ for
 a prime $p$.

For $h \in \mathbb{N}_{\geq 0}$, we now define $h$-solvability of a knot. As indicated in the
introduction, there is an extension of this definition to $h \in \frac{1}{2} \mathbb{N}_{\geq 0}$.
We do not require this more general definition, and refer the reader to~\cite[Definition 1.2]{Cochran-Orr-Teichner:1999-1}
for details.
\begin{defn}\label{defn:solvable}
  A knot $K$ is \emph{$h$-solvable} if there exists a compact spin 4-manifold
  $W$ such that $\partial W= M_K$, the inclusion induced map $H_1(M_K) \to
  H_1(W)$ is an isomorphism, and there exist embedded surfaces with trivial
  normal bundle $D_1, \dots, D_k$ and $L_1, \dots, L_k$ in $W$ such that
  \begin{enumerate}
    \item The surfaces are pairwise disjoint except for $D_j$ and $L_j$, which
    for each $j =1,\dots,k$ intersect transversely in a single point.
    \item The second homology classes represented by $D_1, \dots, D_k, L_1,
    \dots, L_k$ generate $H_2(W)$.
    \item The inclusion induced maps $\pi_1(D_i) \to \pi_1(W)$ and $\pi_1(L_i)
    \to \pi_1(W)$ have image contained in $\pi_1(W)^{(h)}$.
  \end{enumerate}
\end{defn}

This gives a filtration of the knot concordance group by subgroups
$\mathcal{F}_h$ consisting of the concordance classes of $h$-solvable knots,
explored in \cite{Cochran-Orr-Teichner:1999-1, Cochran-Orr-Teichner:2002-1,
Cochran-Teichner:2003-1, Cochran-Harvey-Leidy:2009-1}, among others. Every
$1$-solvable knot is algebraically slice and every $2$-solvable knot has
vanishing Casson-Gordon invariant sliceness obstruction. In particular, as
mentioned in the introduction, the traditional 4-genus lower bounds of
Tristram-Levine and Casson-Gordon signatures cannot be usefully employed with
$2$-solvable knots.

The satellite operation interacts particularly nicely with the solvable
filtration. We remind the reader that given a knot $R$, \emph{infection curves}
$\alpha_1, \dots, \alpha_k$ in $S^3 \ssm \nu(R)$ that form an unlink in $S^3$,
and infection knots $J_1, \dots, J_k$, the \emph{satellite of $R$ by $\{J_i\}$
along $\{\alpha_i\}$} is defined to be the image of $R$ in
\[
  \Big( S^3 \ssm \bigsqcup_{i=1}^k \nu(\alpha_i)\Big) \cup
  \bigsqcup_{i=1}^k E_{J_i} \cong S^3,
\]
where $E_{J_i}$ is the exterior of $J_i$ and the identification is made so that
a $0$-framed longitude of $\alpha_i$, denoted by $\lambda(\alpha_i)$, is
identified with a meridian of $J_i$ and vice versa. We denote this knot
by~$R_{\alpha}(J)$. The next proposition comes from
\cite[Proposition~3.1]{Cochran-Orr-Teichner:2002-1}.  We will apply it with
$h=2$ to see that the knots we construct are $2$-solvable.

\begin{prop}\label{prop:infection2solvable}
  Let $R$ be a slice knot and $\{\alpha_i\}_{i=1}^k$ be a collection of
  unknotted, unlinked curves in $S^3 \ssm R$ such that $[\alpha_i] \in
  \pi_1(M_R)^{(h)}$ for all $i=1, \dots, k$. If for each  $i=1, \dots, k$ the
  knot $J_i$ has $\Arf(J_i)=0$, then  $R_{\alpha}(J)$ is $h$-solvable.
\end{prop}

While our discussions have been thus far focused on the topological category, there are  analogous notions of smooth sliceness, concordance, and 4-genera of knots. There is considerable interest in understanding the structure of $\mathcal{T}$, the collection of topologically slice knots modulo smooth concordance.
Here the $h$-solvable filtration is of no use, since every topologically slice knot lies in $\bigcap_{h=0}^{\infty} \mathcal{F}_h$.  This prompted Cochran-Harvey-Horn to define the \emph{bipolar filtration} as follows.

\begin{defn}\label{defn:positive}
  A knot $K$ is $h$-positive (respectively $h$-negative) if there exists a
  smoothly embedded disc $D$ in a smooth simply connected 4-manifold $V$ such
  that $\partial(V, D)= (S^3, K)$ and such that there exist disjointly embedded
  surfaces $S_1, \dots, S_k$ in $V \ssm \nu(D)$ which form a basis for $H_2(V)$
  such that for each $i=1, \dots k$,
  \begin{enumerate}
    \item The surface $S_i$ has $S_i \cdot S_i=+1$ (respectively $S_i \cdot
    S_i=-1$).
    \item The inclusion induced map $\pi_1(S_i) \to \pi_1(V \ssm D)$ has image
    contained in~$\pi_1(V\ssm D)^{(h)}$.
  \end{enumerate}
\end{defn}

Note that smoothly slice knots are $h$-positive for all $h \in \mathbb{N}$, that
the connected sum of two $h$-positive knots is $h$-positive, and that any knot
that can be unknotted by changing crossings from positive to negative (negative
to positive) is $0$-positive ($0$-negative)~\cite{Cochran-Harvey-Horn:2012-1}.

\begin{defn}
  We say that a knot $K$ is \emph{$h$-bipolar} if it is both $h$-positive and $h$-negative.
\end{defn}

The following proposition, inspired by \cite[Lemma 2.3]{Cha-Kim:2017}, gives us
a way to construct $h$-bipolar knots; we will apply it when $h=2$.

\begin{prop}\label{prop:buildingbipolar}
  Let $R$ be a smoothly slice knot and let $\eta^+$ and $\eta^-$ be curves in
  the complement of $R$ that form an unlink in~$S^3$.   Suppose that each
  $\eta^{\pm}$ represents a class in $\pi_1(S^3 \ssm R)^{(h)}$, and that for
  any knot $J$ we have that $R_{\eta^+}(J)$ and $R_{\eta^-}(J)$ are both
  smoothly slice. Then for any  $0$-positive knot $J^+$  and 0-negative knot
  $J^-$, the satellite knot $R_{\eta^+, \eta^-}(J^+, J^-)$ is $h$-bipolar.
\end{prop}

\begin{proof}
  Since $R_{\eta^+}(J^+)$ is slice and $J^-$ is 0-negative, the knot
  \[
    R_{\eta^+, \eta^-}(J^+, J^-)= \bigl(R_{\eta^+}(J^+)\bigr)_{\eta_-}(J^-)
  \]
  is $h$-negative by \cite[Proposition 3.3]{Cochran-Harvey-Horn:2012-1}. We see
  that
  \[
    R_{\eta^+, \eta^-}(J^+, J^-) = \bigl(R_{\eta^-}(J^-)\bigr)_{\eta_+}(J^+)
  \]
  is $h$-positive by a symmetric argument.
\end{proof}

\section{Disconnected manifolds, fundamental groups and twisted homology}\label{section:disconnected}

We will need to understand the twisted homology of a connected 4-manifold $X$
with disconnected boundary $Y$.  In this section, we establish some technical
details in this setting, for example by defining inclusion maps from the twisted
homology of $Y$ to that of $X$ and showing that there is a long exact sequence of
the homology of the pair $(X, Y)$. On a first reading we encourage the reader to
skim this section, focusing on the paragraph leading into
Definition~\ref{defn:disconnextension} and the statement of
Proposition~\ref{prop:les-pairs}. A similar discussion can be found in
\cite[Section~2.1]{FK06}.

We note once and for all that manifolds are oriented and either compact or
arising as an infinite cover of a compact manifold.  For a manifold $V$, we
write $p \colon \wt{V} \to V$ for the universal cover.

Let $Y = \bigsqcup_{i=1}^N Y_i$ be a compact $\ell$-dimensional manifold with
$N$ connected components.  Let $y_i \in Y_i$ be a basepoint for each connected
component. Let $S$ be a ring with unity and let $A$ be a left $S$-module.  A
representation $\Phi$ of the fundamental groupoid of $Y$ into $\Aut(A)$ is
equivalent to a homomorphism
\[
  \Phi = \coprod_{i=1}^N \Phi_i \colon
  \coprod_{i=1}^N \pi_1(Y_i,y_i) \to \Aut(A)
\]
from the free product of the fundamental groups of the connected components to
$\Aut(A)$. We will use the following examples.

\begin{enumerate}[(a)]

  \item Let $\Gamma$ be a group.  Then we will take $A = S= \Z\Gamma$, with
  $\Phi_i \colon \pi_1(Y_i,y_i) \to \Gamma \subseteq \Aut(A)$, where $g \in
  \Gamma$ acts on $A$ by left multiplication.  We will also take
  $A=\mathcal{N}\Gamma$, the group Von Neumann algebra of $\Gamma$, discussed in
  Section~\ref{section:L2-invariants}.

  \item  The ring $S$ is a commutative PID and $A= S^r$, together with a
  homomorphism
  \[
    \Phi_i \colon \pi_1(Y_i,y_i) \to GL_r(S) = \Aut(S^r).
  \]

\end{enumerate}

For each $i$ we use the representation $\Phi_i\colon \pi_1(Y_i,y_i) \to \Aut(A)$
to give $A$ a right $\Z[\pi_1(Y_i,y_i)]$-module structure.  Then we let
$C_*(\wt{Y}_i)$ be a cellular chain complex obtained by lifting  some CW decomposition of $Y_i$
(or a CW complex homotopy equivalent to $Y_i$ in the case that $Y_i$ is a
topological 4-manifold), and define the homology of $Y$ twisted by $\Phi$ to be
\[
  H_*(Y;A) := \bigoplus_{i=1}^N
  H_*\bigl(A \otimes_{\Z[\pi_1(Y_i,y_i)]} C_*(\wt{Y}_i)\bigr).
\]

Now suppose that $Y = \partial X$, where $X$ is a compact connected
$(\ell+1)$-dimensional manifold with $\partial X=Y = \bigsqcup_{i=1}^N Y_i$. A
schematic of a similar situation is shown in Figure~\ref{fig:diagramzz}.  Let $x
\in X$ be a basepoint and let $\tau_i \colon [0,1] \to X$ be a path from $x$ to
$y_i$. The paths $\tau_i$ induce homomorphisms $\iota_i \colon \pi_1(Y_i,y_i)
\to \pi_1(X,x)$, by $\gamma \mapsto \tau_i \gamma \overline{\tau_i}$.

\begin{defn}\label{defn:disconnextension}
  We say that $\Phi \colon \coprod_{i=1}^N \pi_1(Y_i,y_i) \to \Aut(A)$
  \emph{extends over $X$} if there is a homomorphism $\Psi \colon \pi_1(X,x) \to
  \Aut(A)$ such that $\Psi \circ \iota_i = \Phi_i$ for each $i=1,\dots,N$.
\end{defn}

Use the inclusion $j_i\colon Y_i \to X$ to define the pullback cover of $Y_i$ in
terms of  the universal cover of $X$ via the diagram:
\[
  \xymatrix{
    \wt{Y}_i^X \ar[r] \ar[d] & \wt{X} \ar[d]^p \\
    Y_i \ar[r]_{j_i} & X
  }
\]
The pullback $\wt{Y}_i^X$ is given by pairs $\{(y, \widetilde{x}) \in Y_i \times
\wt{X} \mid j_i(y) = p(\widetilde{x})\}.$  Apply the action of the group
$\pi_1(X,x)$ on $\wt{X}$ to the second factor to obtain an action of
$\pi_1(X,x)$ on $\wt{Y}_i^X$. This is defined since the action on $\wt{X}$ is
equivariant with respect to~$p$. The action of $\pi_1(X,x)$ on $\wt{Y}_i^X$
induces an action of $\Z[\pi_1(X,x)]$ on the chain complex $C_*(\wt{Y}_i^X)$.

\begin{lem}\label{lemma:pullback-covers}
  We have a homeomorphism
  \[
    \pi_1(X,x) \times_{\pi_1(Y,y_i)} \wt{Y}_i \cong \wt{Y}_i^X,
  \]
  where by definition the left hand side means:
  \[
    \pi_1(X, x) \times \wt{Y_i}/ \big((\gamma, \wt{y}) \sim (\gamma', \wt{y}')
    \text{ if there is } g \in \pi_1(Y, y_i) \text{ such that }
    \gamma  \iota_i(g) = \gamma' \text{ and } g \cdot \wt{y}= \wt{y}'\big).
  \]
\end{lem}

\begin{proof}
  Start with the covering space $\pi_1(Y_i,y_i) \to \wt{Y}_i \to Y_i$ with fibre
  $\pi_1(Y_i,y_i)$, and then apply the `product over $\pi_1(Y_i,y_i)$'
  construction to obtain a covering space $\pi_1(X,x) \times_{\pi_1(Y_i,y_i)}
  \pi_1(Y_i,y_i) \to \pi_1(X,x) \times_{\pi_1(Y_i,y_i)} \wt{Y_i} \to Y_i.$ Since
  $\pi_1(X,x) \times_{\pi_1(Y_i,y_i)} \pi_1(Y_i,y_i) \cong \pi_1(X,x)$ as
  discrete spaces and affine sets over $\pi_1(X,x)$, this fibre bundle is
  homeomorphic to
  \[
    \pi_1(X,x) \to \pi_1(X,x) \times_{\pi_1(Y_i,y_i)} \wt{Y_i} \to Y_i.
  \]
  Since both $\pi_1(X,x) \times_{\pi_1(Y_i,y_i)} \wt{Y_i}$ and $\wt{Y}^X$ are
  covering spaces of $Y_i$ corresponding to the homeomorphism $\iota_i \colon
  \pi_1(Y_i,y_i) \to \pi_1(X,x)$, they are homeomorphic by the classification of
  covering spaces.
\end{proof}

It follows from Lemma~\ref{lemma:pullback-covers} that we have an chain
isomorphism $\Z[\pi_1(X,x)] \otimes_{\Z[\pi_1(Y_i,y_i)]}  C_*(\wt{Y}_i) \cong
C_*(\wt{Y}_i^X)$. Consider the sequence of chain maps:
\begin{align*}
  A \otimes_{\Z[\pi_1(Y_i,y_i)]} C_*(\wt{Y}_i) & \xrightarrow{\cong}
  A \otimes_{\Z[\pi_1(X,x)]} \Z[\pi_1(X,x)] \otimes_{\Z[\pi_1(Y_i,y_i)]}
  C_*(\wt{Y}_i)
  \\
  & \xrightarrow{\cong}
  A \otimes_{\Z[\pi_1(X,x)]} C_*(\wt{Y}^X_i)
  \\
  & \xrightarrow{\hphantom{\cong}}
  A \otimes_{\Z[\pi_1(X,x)]} C_*(\wt{X}).
\end{align*}
The first map sends $a \otimes c \mapsto a \otimes 1 \otimes c$. The second map
uses the isomorphism discussed above, and the third map is induced by~$j_i$.

This chain level map induces a map on homology $(j_i)_*\colon H_*(Y_i;A) \to
H_*(X;A)$, which in turn induces
\[
  \bigoplus_{i=1}^N (j_i)_* \colon \bigoplus_{i=1}^N H_*(Y_i;A)
  \cong H_*(Y;A) \to H_*(X;A).
\]
Let $\wt{Y}^X := \bigsqcup_{i=1}^N \wt{Y}_i^X$. Then by identifying $\wt{Y}^X$
with its image in $\wt{X}$ we also have relative twisted homology groups
\[
  H_*(X,Y;A) := H_*(A \otimes_{\Z[\pi_1(X,x)]} C_*(\wt{X},\wt{Y}^X)).
\]
The chain maps above fit into a short exact sequence of chain complexes
\[
  0 \to \bigoplus_{i=1}^N A \otimes_{\Z[\pi_1(Y_i,y_i)]} C_*(\wt{Y}_i)
  \to A \otimes_{\Z[\pi_1(X,x)]} C_*(\wt{X})
  \to  A \otimes_{\Z[\pi_1(X,x)]} C_*(\wt{X},\wt{Y}^X) \to 0.
\]
That this is exact follows from the chain isomorphism
\[
  A \otimes_{\Z[\pi_1(Y_i,y_i)]} C_*(\wt{Y}_i)
  \cong A \otimes_{\Z[\pi_1(X,x)]} C_*(\wt{Y}^X_i).
\]
The short exact sequence of chain complexes gives rise to a long exact sequence
in homology, which we record in the next proposition.

\begin{prop}\label{prop:les-pairs}
  With a fixed choice of paths $\{\tau_i\}$ and a representation $\Phi\colon
  \coprod_{i=1}^n \pi_1(Y_i, y_i) \to \Aut(A)$ that extends over $X$, there is a
  long exact sequence in twisted homology
  \[
    \cdots \to H_k(Y;A) \to H_k(X;A) \to H_k(X,Y;A) \to H_{k-1}(Y;A) \to \cdots
  \]
  with $H_k(Y;A) \to H_k(X;A)$ the inclusion induced map discussed above.
\end{prop}

In later sections  we work with many different representations of a given
fundamental group(oid), and so we emphasize the representation $\Phi$ rather
than the module $A$ by writing $H_k^\Phi(Y)$ for $H_k(Y; A)$.

\section{$L^{(2)}$-signature invariants}\label{section:L2-invariants}

In this section we introduce the Von Neumann $L^{(2)}$ $\rho$-invariant of a
closed (not necessarily connected) 3-manifold equipped with a representation of
its fundamental group or groupoid, and we recall the key properties of this
invariant required for the proof of Theorem~\ref{thm:mainthm-intro}.  In
particular, we review the Cheeger-Gromov bound, a satellite formula, and an
upper bound in terms of the second Betti number of a bounding $4$-manifold.

\begin{defn}\label{defn:rho-inv}
  Let $Y$ be a closed oriented 3-manifold, let $\Gamma$ be a discrete group, and
  let $\phi \colon \pi_1(Y) \to \Gamma$ be a representation. Note that $Y$ might
  be disconnected, in which case we use the conventions of
  Section~\ref{section:disconnected}. Suppose that $\phi$ extends to $\Phi\colon
  \pi_1(W) \to \Gamma$ where $W$ is a compact oriented 4-manifold with $\partial
  W = Y$. The \emph{von Neumann $L^{(2)}$ $\rho$-invariant} of $(Y, \phi)$ is
  the signature defect
  \[
    \rho^{(2)}(Y,\phi) = \sigma^{(2)}_\Gamma(W, \Phi)- \sigma(W),
    \]
  where $\sigma^{(2)}_\Gamma(W, \Phi)$ is the $L^{(2)}$-signature of the
  intersection form  $\lambda_{\Gamma} \colon H_2(W, \mathcal{N}\Gamma) \times
  H_2(W, \mathcal{N}\Gamma) \to \mathcal{N}\Gamma$ and $\sigma(W)$ is the
  ordinary signature of the intersection form on $H_2(W;\Q)$.  Here the
  $L^{(2)}$-signature is defined via the completion $\Z\Gamma \to
  \mathbb{C}\Gamma \to \mathcal{N}\Gamma$ to the Von Neumann algebra, and the
  spectral theory of operators on $\mathcal{N}\Gamma$-modules. We refer to
  \cite[Section~5]{Cochran-Orr-Teichner:1999-1} and
  \cite[Section~3.1]{Cha:2014-1} for more details. In particular,
  $\rho^{(2)}(Y,\phi)$ only depends on the pair $(Y,\phi)$  since both the
  $L^{(2)}$ signature and the ordinary signature satisfy Novikov additivity and
  also $\sigma^{(2)}_\Gamma(V, \Phi)= \sigma(V)$ for a \emph{closed}
  4-manifold~$V$. (See~\cite[p.~323]{Chang-Weinberger:2003-1}
  and~\cite[Lemma~5.9]{Cochran-Orr-Teichner:1999-1}.)
\end{defn}

This invariant was originally defined by Cheeger and Gromov via Riemannian
geometry and $\eta$-invariants, independently of any bounding 4-manifold, so the
above definition could be taken as a proposition that the two definitions
coincide. For our purposes, as is common in the knot concordance literature, it
is simpler to take the above as the definition; for a discussion, see
\cite[Section~5]{Cochran-Orr-Teichner:1999-1} and
\cite{Cochran-Teichner:2003-1}.

\begin{exl}\label{example:abelian-rep-integral}
  Let $M_J$ be the zero-framed surgery manifold of a knot $J \subset S^3$ and
  let $\phi \colon \pi_1(M_J) \to \Z$ be the abelianization map. Then
  \[\rho_0(J):= \rho^{(2)}(M_J,\Z) = \int_{\omega \in S^1} \sigma_{\omega}(J)
  \,d \omega,\] where $\sigma_{\omega}(J)$ is the Tristram-Levine signature of
  $J$ at $\omega \in S^1$, that is the signature of $(1-\omega)V +
  \overline{(1-\omega)}V^T$ for $V$ a Seifert matrix of~$J$.  See
  \cite[Lemma~5.4]{Cochran-Orr-Teichner:1999-1} for the proof.
\end{exl}

We will need the following theorem of Cheeger and Gromov, establishing a
universal bound for the $\rho$-invariants of a fixed closed 3-manifold $Y$.

\begin{thm}[\cite{Cheeger-Gromov:1985-1}]\label{thm:cheeger-gromov-thm}
  Let $Y$ be a closed oriented 3-manifold. Then there exists a constant $C$ such
  that $ | \rho^{(2)}(Y, \phi) | \leq C$ for any discrete group $\Gamma$ and any
  representation $\phi \colon \pi_1(Y) \to \Gamma$.
\end{thm}

We will refer to the infimum of all such constants $C$ as the Cheeger-Gromov
constant of $Y$, denoted $C(Y)$. We note that \cite{Cha:2016-CG-bounds} has
given a proof of Theorem~\ref{thm:cheeger-gromov-thm} using the signature defect
definition of $\rho^{(2)}(Y,\phi)$ given above, and has given explicit bounds
for $C(Y)$ in terms of the triangulation complexity of~$Y$.

The following proposition comes from~\cite{Cochran-Harvey-Leidy:2009-1}.

\begin{prop}\label{prop:additivity}
  Let $K= R_{\alpha}(J)$ be the result of a satellite operation on a knot $R$ by
  infection knots $\{J_k\}$ along infection curves $\{\alpha_k\}$. Let $\phi
  \colon \pi_1(M_K) \to \Gamma$, and suppose that  for some $h \in \mathbb{N}$
  we have $\alpha_k \in \pi_1(M_R)^{(h)}$ for all $k$ and  $\Gamma^{(h+1)}=1$.
  Suppose that for all $k$, either $\phi(\alpha_k) = 1$ or $\phi(\alpha_k)$ is
  infinite order in $\Gamma$. Then the restriction induced maps $\pi_1(M_R \ssm
  \bigsqcup  \nu(\alpha_k)) \to \Gamma$ and $\pi_1(E_{J_k}) \to \Gamma$ extend
  uniquely to  $\phi_0 \colon \pi_1(M_R) \to \Gamma$ and $\phi_k\colon
  \pi_1(M_{J_k}) \to \Gamma$ and we have
  \begin{align*}
    \rho^{(2)}(M_K, \phi)&= \rho^{(2)}(M_R, \phi_0)+
    \sum_{k} \rho^{(2)}(M_{J_k}, \phi_k).
  \end{align*}
\end{prop}

\begin{proof}
  The proof of \cite[Lemma~2.3]{Cochran-Harvey-Leidy:2009-1} applies, with the
  following modification. The original statement of this proposition assumes the
  additional hypothesis that $\Gamma$ is a poly-torsion-free-abelian (PTFA)
  group. However, an inspection of the proof shows that we need only assume that
  for each $k$ either $\phi(\alpha_k) = 1 \in \Gamma$ or $\phi(\alpha_k)$ is
  infinite order.  In the case that $\phi(\alpha_k) \neq 1$ in $\Gamma$, they
  need in the proof of \cite[Lemma~2.3]{Cochran-Harvey-Leidy:2009-1} that
  $H_1(\alpha_k;\Z\Gamma) =0$.  But since $\phi(\alpha_k)$ is infinite order,
  $H_1(\alpha_k;\Z\Gamma)$ is the first homology of $\mathbb{R} \times
  \Gamma/\langle \alpha_k \rangle$, which vanishes.
\end{proof}

We will apply Proposition~\ref{prop:additivity} when $\Gamma= G/ G^{(3)}_{(\Q,
\Z_p, \Q)}$ for some group $G$ and $h=2$. For such $\Gamma$, any curves
$\alpha_k \in \pi_1(M_R)^{(2)}$ satisfy the hypothesis of the proposition, since
then $\phi(\alpha_k) \in \Gamma^{(2)}$ and $\Gamma^{(2)}/ \Gamma^{(3)}$ is
torsion-free.

Under the assumptions of Proposition~\ref{prop:additivity}, we have that the map
$\phi_k \colon \pi_1(M_{J_k}) \to \Gamma$ factors through the abelianization
map. To see this, note that each meridian of $J_k$ is identified with a
longitude of $\alpha_k$, which lies in $\pi_1(M_R)^{(h)}$ and hence is sent to
$\Gamma^{(h)}$. So the image of $\phi_k\colon \pi_1(M_{J^k})\to \Gamma$ is
contained in $\Gamma^{(h)}$, which is an abelian group since $\Gamma^{(h+1)}=1$.
When $\Gamma^{(h)}/ \Gamma^{(h+1)}$ is torsion-free, as occurs when $\Gamma= G/
G^{(3)}_{(\Q, \Z_p, \Q)}$ and $h=2$, we therefore have that $\phi_k$ is either
the zero map or maps onto a copy of $\Z$ in $\Gamma$. By the principle of
$L^{(2)}$-induction \cite[Proposition~5.13]{Cochran-Orr-Teichner:1999-1} and
Example~\ref{example:abelian-rep-integral}, we have that

\begin{align*}
  \rho^{(2)}(M_{J_k}, \phi_k)= \begin{cases}
  \rho_0(J_k) & \text{ if } \phi_k \neq 0 \\
  0 & \text{ if } \phi_k =0.
  \end{cases}
\end{align*}
Finally, note that since a meridian of $J_k$ is identified with a longitude
$\lambda(\alpha_k)$ of $\alpha_k$ in $M_K$, we have that $\phi_k$ is the zero
map if and only if $\phi(\lambda(\alpha_k))=0$. We summarize the results of the
above discussion for later use.

\begin{prop}\label{prop:ouradditivity}
  Let $K=R_{\alpha}(J)$ be the result of a satellite operation on $R$ by
  infection knots $\{J_k\}$ along infection curves $\{\alpha_k\}$ lying
  in~$\pi_1(M_R)^{(2)}$. Let $\Gamma= G/ G_{(\Q, \Z_p, \Q)}^{(3)}$ for some
  group $G$ and prime~$p$, and let $\phi\colon \pi_1(M_K) \to \Gamma$. Then the
  restriction induced maps $\pi_1(M_R \ssm \bigsqcup \nu(\alpha_k)) \to \Gamma$
  and $\pi_1(E_{J_k}) \to \Gamma$ extend uniquely to maps $\phi_0\colon
  \pi_1(M_R) \to \Gamma$ and $\phi_k\colon \pi_1(M_{J_k}) \to \Gamma$. Moreover,
  \[
    \rho^{(2)}(M_K, \phi) = \rho^{(2)}(M_R, \psi_0) +
    \sum_k \rho^{(2)}(M_{J_k}, \phi_k) =\rho^{(2)}(M_R, \phi_0) +
    \sum_k \delta_k(\psi) \rho_0(J_k),
  \]
  where
  \[
    \delta_k(\psi)= \begin{cases} 0  &\text{ if }  \psi(\lambda(\alpha_k))=0 \\
    1 &\text{ if }  \psi(\lambda(\alpha_k)) \neq 0. \end{cases}
  \]
\end{prop}

The following straightforward consequence of \cite[Theorem~3.11]{Cha:2014-1}
will provide our key upper bound on $L^{(2)}$-signatures.  Strebel's class of
groups $D(\Z_p)$ was defined in \cite{Strebel:1974-1}; we will not recall the
definition.   We will use the fact that for any group $G$ and any $h \in
\mathbb{N}$, we have that $\Lambda = G/G^{(h)}_\mathcal{S}$ is amenable and lies
in $D(\Z_p)$ provided $S_i$ is either $\Q$ or $\Z_{p^{a_i}}$ for every $i \in
\mathbb{N}$~\cite[Lemma~6.8]{Cha-Orr:2009-01}.

\begin{thm}\label{thm:upperbound}
  Let $Z$ be a 4-manifold with boundary $\partial Z=Y$ and let $\phi \colon
  \pi_1(Y) \to \pi_1(Z) \to \Lambda$ be a homomorphism, where $\Lambda$ is
  amenable and in Strebel's class $D(\mathbb{Z}_p)$. Then $|\rho^{(2)}(Y, \phi)|
  \leq 2 \dim_{\Z_p} H_2(Z, \mathbb{Z}_p).$
\end{thm}

\begin{proof}[Proof of Theorem~\ref{thm:upperbound}]

  Let $\widetilde{Z}$ be the cover of $Z$ induced by the homomorphism $\pi_1(Z)
  \to \Lambda$. Since $Z$ is a compact 4-manifold with boundary, it has the
  homotopy type of a finite 3-dimensional CW complex. This follows from
  \cite[\textsection 1(III)]{KS69} to get a finite CW complex, combined with
  \cite[Corollary 5.1]{Wa66} to restrict the dimension of the CW complex to
  three. Let $C_*$ be the corresponding chain complex, and let
  $\widetilde{C_*}:= C_*(\widetilde{Z})$ denote the chain complex of
  $\widetilde{Z}$. Since $\Lambda$ is amenable and in
  Strebel's~\cite{Strebel:1974-1} class $D(\Z_p)$,
  \cite[Theorem~3.11]{Cha:2014-1} tells us that
  \begin{align*}
    \dim^{(2)} H_2(Z; \mathcal{N} \Lambda)
    & = \dim_{\Z_p} H_2( \Z_p \otimes_{\Z \Lambda} \widetilde{C_*})
    \\
    &= \dim_{\Z_p} H_2( \Z_p \otimes_{\Z} C_*)= \dim_{\Z_p} H_2(Z; \Z_p).
  \end{align*}
  It follows that
  \begin{align*}
    |\rho^{(2)}(Y, \phi)|&= | \sigma_\Lambda^{(2)}(Z)- \sigma(Z)|
    \leq \dim^{(2)} H_2(Z; \mathcal{N} \Lambda) + \dim_{\Q} H_2(Z;\Q)
    \leq 2\dim_{\Z_p} H_2(Z; \Z_p).
  \end{align*}
  We use the universal coefficient theorem to deduce that $\dim_{\Q} H_2(Z;\Q)
  \leq \dim_{\Z_p} H_2(Z; \Z_p)$ for the final inequality.
\end{proof}

\section{Metabelian twisted homology}\label{section:metabelian-homology}

In this section we review Casson-Gordon type metabelian representations of knot
groups, and the resulting twisted homology.  The behavior of infection curves in
this twisted homology will be key to our proof of
Theorem~\ref{thm:mainthm-intro}.

We now let $S$ denote a commutative PID  and let $Q$ denote its quotient field.
We will often take  $S= \F[t^{\pm1}]$ and $Q= \F(t)$  for some field $\F$, as
well as $S= \Z$ and $Q=\Q$.

Let $X$ be a space homotopy equivalent to a finite CW-complex and let $A$ be a
left $S$-module given the structure of a right $\Z[\pi_1(X)]$-module by a
homomorphism $\phi\colon \pi_1(X) \to \Aut(A)$. Note that $\pi_1(X)$ naturally
acts on $\widetilde{C_*}$, the chain complex of the universal cover
$\widetilde{X}$ of $X$, on the left. Then as in
Section~\ref{section:disconnected}, the twisted homology $H_*(X; A)$ is defined
to be
\begin{align*}
  H_*^\phi (X):= H_*(A \otimes_{\Z[ \pi_1(X)]} \widetilde{C_*}).
\end{align*}

We will be particularly interested in the following metabelian representations.
Suppose that we have a preferred surjection $\varepsilon \colon H_1(X) \to \Z$.
For every $r \in \mathbb{N}$, we let $p_r\colon \Z \to \Z_r$ be the usual
projection map and let $X^r$ denote the $r$-fold cyclic cover of $X$
corresponding to $\ker(p_r \circ \varepsilon)$. Note that covering
transformations give $H_1(X^r)$ the structure of a
$\mathbb{Z}[\mathbb{Z}_r]$-module.
 Choosing  a preferred element $\gamma_0 \in \pi_1(X)$ with
$\varepsilon(\gamma_0)=+1$ then gives us a map
\begin{align*}
  \rho_{\gamma_0}\colon \pi_1(X) \to \Z \ltimes H_1(X^r) \quad\text{by }
  \gamma \mapsto \big(t^{\varepsilon(\gamma)},
  [\gamma_0^{-\varepsilon(\gamma)} \gamma]\big),
\end{align*}
where $\gamma_0^{-\varepsilon(\gamma)} \gamma \in \pi_1(X_r) \leq \pi_1(X)$ and
$[\gamma_0^{-\varepsilon(\gamma)} \gamma]$ denotes the image of
$\gamma_0^{-\varepsilon(\gamma)} \gamma$ under the Hurewicz map.

Given any choice of a homomorphism $\chi\colon H_1(X^r) \to \Z_m$, we let
$\xi_m= e^{2 \pi i/ m}$ and  obtain a map $\theta_\chi \colon \Z \ltimes
H_1(X^r) \to \GL_r(\Q(\xi_m)[t^{\pm1}])$ by
\[
  (t^j,a) \mapsto
  \begin{bmatrix}
    0& \dots &0&t \vphantom{\xi_m^{\chi(a)}}\\
    1&0& \dots &0 \vphantom{\xi_m^{\chi(a)}}\\
    \vdots &\ddots & \ddots & \vdots \\
    0&\dots & 1 & 0 \vphantom{\xi_m^{\chi(a)}}
  \end{bmatrix}^j
  \begin{bmatrix}
    \xi_m^{\chi(a)} & 0 & \dots & 0 \\
    0&\xi_m^{\chi(t\cdot a)} & \dots& 0\\
    \vdots&\vdots&\ddots & \vdots\\ 0&0&\dots &\xi_m^{\chi(t^{k-1} \cdot a)}
  \end{bmatrix}.
\]
We then let $S=\Q(\xi_m)[t^{\pm1}]$ and $A= \Q(\xi_m)[t^{\pm1}]^r$, noting that
 $\theta_\chi \circ \rho_{\gamma_0}$ gives $A$ a right $\Z[\pi_1(X)]$-module
 structure. These representations appear in \cite{Kirk-Livingston:1999-2, Let00,
 Friedl:2003-4, HKL08, Friedl-Powell:2010-1}, modelled on the covering spaces
 used in the definition of Casson-Gordon invariants~\cite{Casson-Gordon:1986-1}.
 We refer to such representations as \emph{Casson-Gordon type representations}.

In particular, given an oriented knot $K$ and a preferred meridian $\mu \in
\pi_1(X_K)$, the canonical abelianization map $\varepsilon\colon \pi_1(X_K) \to
\Z$ has $\varepsilon(\mu)= +1$. Note that since the zero-framed longitude
$\lambda_K$ of~$K$ is an element of $\pi_1(X_K)^{(2)}$, for every $r \in
\mathbb{N}$ the map $\rho_\mu\colon \pi_1(X_K) \to \Z \ltimes H_1(X^r_K)$
extends uniquely over $\pi_1(M_K)$. The homology $H_1(X^r)$ splits canonically
as $H_1(\Sigma_r(K)) \oplus \Z$, where $\Sigma_r(K)$ is the $r$th cyclic
branched cover of $S^3$ along $K$. Our map $\chi\colon H_1(X^r) \to \Z_m$ will
always be chosen to factor through the projection map to $H_1(\Sigma_r(K))$.

In the case $r=2$ we have that $t$ must act by $-1$ on $H_1(\Sigma_2(K))~$, as discussed in the first paragraphs of~\cite{Davis95}, and so
we can conveniently decompose $\theta_\chi \circ \rho_\mu$ differently as
$\theta \circ f_\chi$, where
\begin{align*}
  f_\chi \colon \pi_1(M_K) &\to \Z \ltimes \Z_m \\
  \gamma &
  \mapsto  (t^{\varepsilon(\gamma)}, \chi([\mu^{-\varepsilon(\gamma)} \gamma]))
\end{align*}
and
\begin{align*}
  \theta\colon \Z \ltimes \Z_m &\to \GL_2(\Q(\xi_m)[t^{\pm1}]) \\
  (t^j, a) &\mapsto
  \begin{bmatrix}
    0 & t \\
    1 & 0 \\
  \end{bmatrix}^j
  \begin{bmatrix}
    \xi_m^{a} & 0 \\
    0 & \xi_m^{-a}
  \end{bmatrix}.
\end{align*}

The following proposition is a slight modification of a result of
\cite[Prop.~7.1]{MilPow17}, and gives the key connection between a certain
derived series and metabelian homology, when $m=q^s$ is a prime power.

\begin{prop}\label{prop:7-1-MP}
  Let $W$ be a 4-manifold with boundary $\partial W= Y$. Let $\Phi \colon
  \pi_1(W) \to \Aut(A)$ be a representation  that factors through $\Z \ltimes
  \Z_{q^s}$ for some prime $q$, and let $\phi\colon \pi_1(Y) \to \Aut(A)$ be the
  composition of $\Phi$ with the inclusion map $\pi_1(Y) \to \pi_1(W)$. Let
  $\eta \in \pi_1(Y)^{(2)}$ and suppose $\eta$ is sent to the identity in
  $\pi_1(W)/\pi_1(W)^{(3)}_{(\Q, \Z_{q^s}, \Q)}$. Then for any $v \in A$ and any
  $\widetilde{\eta}$, a lift of $\eta$ to the cover of $W$ induced by $\Phi$, we
  have that the class $[v \otimes \widetilde{\eta}]$  in $H_1^{\phi}(Y)$ maps to
  $0$ in $H_1^{\Phi}(W)$.
\end{prop}

\begin{proof}
  The proof of \cite[Prop.~7.1]{MilPow17} (with its first and last sentences
  deleted) applies verbatim.
\end{proof}

Finally, we recall the twisted Blanchfield form. In analogy to the linking form
on the torsion part of the ordinary first homology of a closed oriented
3-manifold, if $\phi\colon \pi_1(M_K) \to GL_r(\Q(\xi_m)[t^{\pm1}])$ arises as
above then there is a metabelian twisted Blanchfield form~\cite{MilPow17}
\[
  \Bl^{\phi}\colon H_1^{\phi}(M_K) \times H_1^{\phi}(M_K)
  \to \Q(\xi_m)(t)/ \Q(\xi_m)[t^{\pm1}].
\]
Note that in the above circumstance  $H_1^\phi(M_K)$ is a torsion
$\Q(\xi_m)[t^{\pm1}]$-module, by the corollary
to~\cite[Lemma~4]{Casson-Gordon:1986-1}; see also~\cite{Friedl-Powell:2010-1}.
In Section~\ref{subsection:example2solvable},  we will need to know that this
form is  \emph{sesquilinear}~\cite{Powell:2016-1}. That is, letting
$\widebar{\cdot}$ denote the involution of $\Q(\xi_m)[t^{\pm1}]$ induced by
sending $t \to t^{-1}$ and $a+bi \mapsto a-bi$,  we have
\[
  \Bl^{\phi}(px, qy)= p \widebar{q} \Bl^{\phi}(x,y)
  \quad\text{for every } p, q \in \Q(\xi_m)[t^{\pm1}]
  \text{ and } x, y \in H_1^{\phi}(M_K).
\]

\section{Main theorem and examples}\label{section:statement}

Here is the result that we use to show that certain satellite knots have large 4-genus.

\begin{thm}\label{thm:general}
  Let $R$ be a ribbon knot and let $\eta^1, \dots, \eta^r$ be curves in $S^3
  \ssm \nu(R)$ that form an unlink in $S^3$ such that each $\eta^j$ represents
  an element of $\pi_1(M_R)^{(2)}$. Suppose that there is a prime $p$ such that
  for every nontrivial character  $\chi\colon H_1(\Sigma_2(R)) \to \Z_p$ we have
  \begin{enumerate}[\upshape\selectfont(1)]
    \item The module $H_1^{\theta \circ f_{\chi}}(M_R):= H_1 \left(\Q(\xi_p)[t^{\pm1}]^2  \otimes_{\Z[\pi_1(M_R)]} C_*(\widetilde{M_R}) \right)$ is nontrivial and
    generated by the collection  $\{[1,0] \otimes [\eta^j]\}_{j=1}^r$.
    \item The order of $H_1^{\theta \circ f_{\chi}}(M_R)$ is relatively prime to
    $\Delta_R(t)$ over $\Q(\xi_{p^a})(t)$ for all $a>0$.
  \end{enumerate}
Let $m_R>0$ denote the generating rank of the $p$-primary part of $H_1(\Sigma_2(R))$ and let $d_R$ denote the number of distinct orders of $H_1^{\theta \circ f_{\chi}}(M_R)$ as $\chi$ ranges over all nontrivial characters from $H_1(\Sigma_2(R))$ to $\Z_p$.

Now fix $g>0$ and suppose that $N\geq \frac{4g(d_R+1)+2}{m_R}$ and
  that the collection of knots $\{J_i^j\mid 1 \leq i \leq N, 1 \leq j \leq r\}$
  satisfy
  \begin{align*}
    |\rho_0(J^j_i)|> 2(2g+N-1) + N C_R
    + \sum_{k=1}^{i-1} \sum_{\ell=1}^r |\rho_0(J^\ell_k)|
    + \sum_{\ell=1}^{j-1} |\rho_0(J^\ell_i)|
  \end{align*}
  for each $1 \leq i \leq N$, $1 \leq  j \leq r$. Then the knot $K= \#_{i=1}^N
  R_{\eta^1, \dots, \eta^r}(J^1_i, \dots, J^r_i)$ has 4-genus at least $g+1$.
\end{thm}
We remark that both $m_R$ and $d_R$ depend not only on the ribbon knot $R$ but also on the choice of prime $p$, though for convenience we suppress this from the notation.  We remark also that
Theorem~\ref{thm:general} can be generalized to consider higher prime power branched covers by appropriately changing the constants; we leave the details of that to the interested reader.

For convenience, we write $\eta^j_i$ for the curve $\eta^j$ in the $i$th copy of
$R$ in $\#_{i=1}^N R$. Note that we can also write $K= \#_{i=1}^N K_i$, where
$K_i:= R_{\eta^1, \dots, \eta^j}(J^1_i, \dots, J^r_i)$. We will prove
Theorem~\ref{thm:general} in Section~\ref{section:proof-of-main-theorem} by
assuming that $g_4(K) \leq g$ and obtaining a contradiction as follows.

Under the assumption that $g_4(K) \leq g$,  we construct a manifold $Z$ with $\partial Z=Y:= \bigsqcup_{i=1}^{N}M_{K_i}$, $\chi(Z)=2g$, and a few other nice properties.  We then let
 \[\psi \colon\pi_1(Y) \to \pi_1(Z) \to \Lambda:= \pi_1(Z)/ \pi_1(Z)^{(3)}_{(\Q, \Z_{p^a}, \Q)}\] be the map induced by inclusion. Since $\Lambda$ is amenable and in $D(\Z_p)$~\cite[Lemma~4.3]{Cha:2014-1}, Proposition~\ref{thm:upperbound} gives  an upper bound on $|\rho^{(2)}(Y, \psi)|$ in terms of $g$. Our result follows from  obtaining a contradictory lower bound on $|\rho^{(2)}(Y, \psi)|$. By Proposition~\ref{prop:ouradditivity} and our choices of the $J^j_i$ knots, we will obtain a contradiction if for some $i$ and $j$ we have  $\lambda(\eta_i^j) \not \in  \pi_1(Z)^{(3)}_{(\Q, \Z_{p^a}, \Q)}$. By Proposition~\ref{prop:7-1-MP}, this will be implied if we can construct some representation $\phi \colon \pi_1(Y) \to \Z \ltimes \Z_{p^a}$ which extends over $\pi_1(Z)$ to $\Phi$ such that for some $i$ and $j$, the
element $[1,0] \otimes [\lambda(\eta_i^j)] \in H_1^{\phi}(Y)$ is not in the kernel of the inclusion induced map $H_1^{\phi}(Y) \to H_1^{\Phi}(Z)$.
The technical work of the proof consists of showing that such a map $\phi$ must exist under the assumption that $g_4(K) \leq g$ together with our construction of $K$.\\

%

We will first give some examples of knots satisfying the hypotheses of the
theorem and then prove some technical lemmas in the next two sections. In
particular we will need to gain some control over the size of certain homology
groups, in order to show that some curve $\eta_i^j$ always survives into a
suitable 3-solvable quotient of the fundamental group of the complement of a
hypothesized locally flat embedded surface of genus~$g$. Of course
Theorem~\ref{thm:mainthm-intro} follows immediately from
Theorem~\ref{thm:general} together with the examples exhibited in
Section~\ref{subsection:example2bipolar} below and (for the grope bounding result) Proposition~\ref{prop:gropebounding}.

It is relatively easy to find examples of seed ribbon knots satisfying the
hypotheses of Theorem~\ref{thm:general}, at least with the help of a computer
program to compute twisted metabelian homology, as developed in \cite{MilPow17}.
In Section~\ref{subsection:example2solvable} we give one such example of a pair
$(R,\alpha)$, and describe the appropriate infection by knots with Arf invariant
zero and large signature in order to obtain 2-solvable large 4-genus knots.

It is a little harder to find  suitable seed knots $R$ that also satisfy
Proposition~\ref{prop:buildingbipolar}, and therefore produce 2-solvable and
2-bipolar large 4-genus examples for the proof of
Theorem~\ref{thm:mainthm-intro}..  Nonetheless, we exhibit such a seed knot $R$
with suitable infection curves in Section~\ref{subsection:example2bipolar}, and
describe the appropriate infection by 0-bipolar knots with large signature in
order to obtain 2-bipolar, 2-solvable large 4-genus knots.

\subsection{Example 1: a 2-solvable knot with large 4-genus}
\label{subsection:example2solvable}

Let $R$ denote the ribbon knot $8_8$, with the unknotted curve $\alpha$ in $S^3
\ssm \nu(R)$ illustrated in Figure~\ref{fig:88}.

\begin{figure}[ht]
  \includegraphics[angle=90, height=5cm]{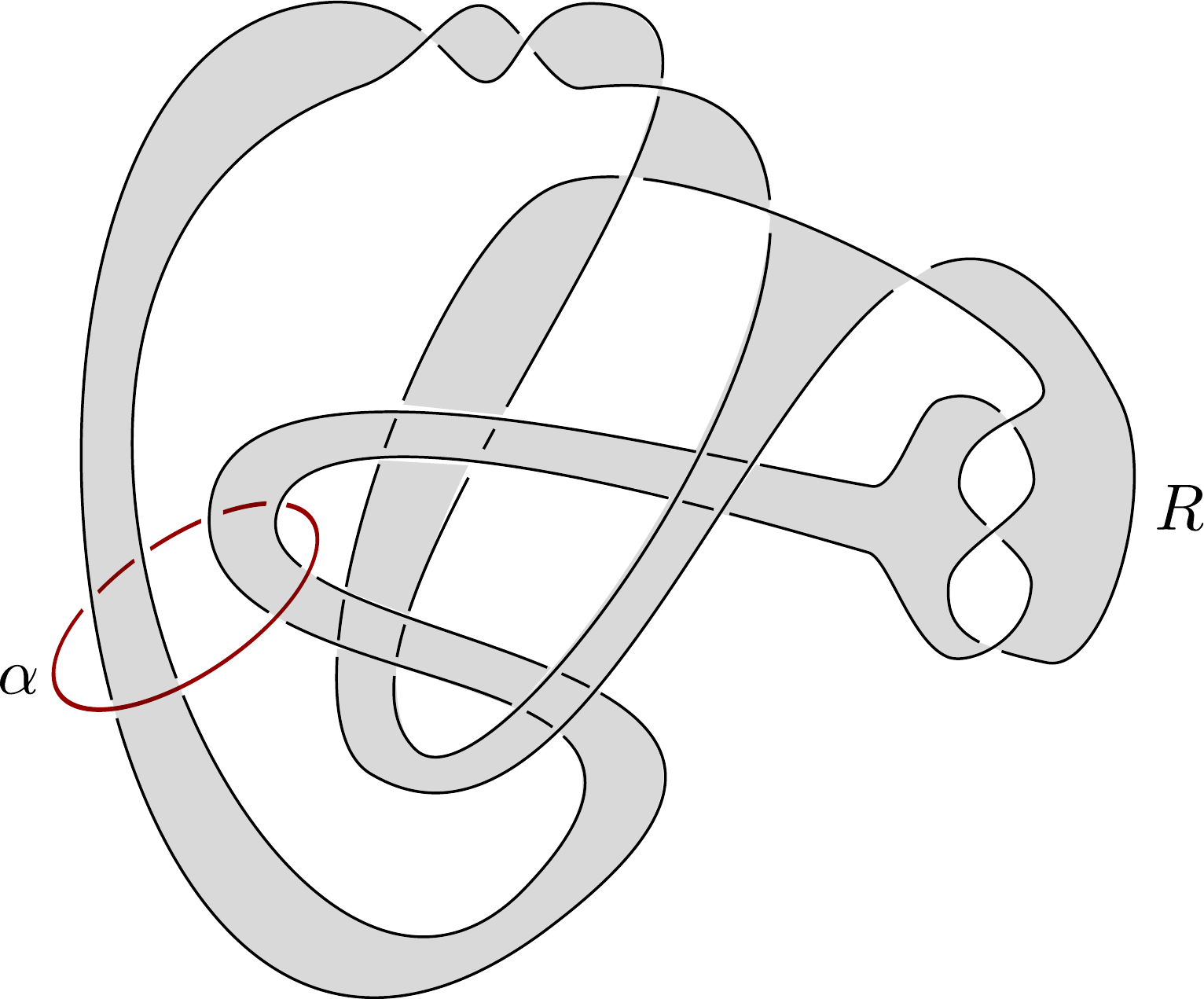}
  \caption{The knot $R=8_8$ with a grey Seifert surface and a red
    infection curve~$\alpha$.}
  \label{fig:88}
\end{figure}

This is the same knot-curve pair $(R,\alpha)$ as in
 \cite[Example~8.1]{MilPow17}, with a slight isotopy to make it more obvious
 that the curve $\alpha$ bounds a surface in the complement of a Seifert surface
 for $R$, and hence lies in $\pi_1(M_R)^{(2)}$. We will need a few computations
 from that paper. First, $H_1(\sr) \cong \Z_{25}$. Note that rescaling a
 character $\chi: H_1(\sr) \to \Z_5$ by a nonzero element of $\Z_5$ does not
 change the underlying $(\Z \ltimes \Z_5)$-covering space of $M_{R}$, and hence
 preserves the isomorphism type of the twisted homology. It is therefore not
 hard to check that given any nontrivial character  $\chi\colon H_1(\Sig_2(R))
 \to \Z_5$, the corresponding twisted homology $H_1^{\theta \circ
 f_\chi}(M_{R})$ is  isomorphic to $\Q(\xi_5)[t^{\pm1}]/ \langle t^2-3t+1
 \rangle$.  Since $\alpha$ is in $\pi_1(M_R)^{(2)}$, it  lifts to a curve
 $\widetilde{\alpha}$ in the $f_\chi$ covering space of $M_R$, and hence
 $x:=[1,0] \otimes \widetilde{\alpha}$ is an element of $H_1^{\theta \circ
 f_\chi}(M_{R})$. Finally, the metabelian twisted Blanchfield pairing
 $Bl^{\theta \circ f_\chi}(x, x)$ is non-zero in $\Q(\xi_5)(t)/
 \Q(\xi_5)[t^{\pm1}]$, as was computed in  \cite[Example~8.1]{MilPow17}.

\begin{lem}
   The element $x$ generates $H_1^{\theta\circ f_{\chi}}(M_R)$.
\end{lem}

 \begin{proof}
  Supposing for a contradiction that $x$ does not generate. Let $y$ be some
  generator for $H_1^{\theta\circ f_{\chi}}(M_R)$. Note that $\Bl^{\theta \circ
  f_\chi}(y,y)$ is of the form $\frac{q(t)}{t^2-3t+1}$ for some $q(t) \in
  \Q(\xi_5)[t^{\pm1}]$.   The polynomial  $t^2-3t+1$ factors as $(t-w_+)(t-w_-)$
  for $w_{\pm}=\frac{3\pm \sqrt{5}}{2} \in \Q(\xi_5)$. Therefore, if $x$ does
  not generate then it must be homologous to $c(t-w_{*}) y$ for some $c \in
  \Q(\xi_5)[t^{\pm1}]$ and $*\in \{ \pm\}$; without loss of generality, say
  $*=+$. But then we can obtain a contradiction, since
  \begin{align*}
    \Bl^{\theta \circ f_\chi}(x,x)&= \Bl^{\theta \circ f_\chi}( c(t-w_{+}) y,  c(t-w_{+}) y)\\
    &= (t-w_+)\overline{(t-w_+)} c \overline{c} \Bl^{\theta \circ f_\chi}(  y,  y) \\
    &= (t-w_+)(t^{-1}-w_+)c \overline{c} \frac{q(t)}{t^2-3t+1} \\
    &=-t^{-1}w_+(t-w_+)(t-w_-) c \overline{c} \frac{q(t)}{t^2-3t+1} \\
    &= -t^{-1} w_+ c \overline{c} q(t)=0 \in \Q(\xi)(t)/ \Q(\xi_5)[t^{\pm1}].\qedhere
  \end{align*}
\end{proof}

Now, fix some $g \geq 0$ and let $N=8g+2$, noting that $m_R$, the generating rank of the 5-primary part of $H_1(\Sigma_2(R))$, is 1 and that there is only one isomorphism class of $H_1^{\theta \circ f_\chi}(M_R)$ and so $d_R=1$ as well. Note that $t^2-3t+1$ is relatively prime to $\Delta_R(t)=
2-6t+9t^2-6t^3+2t^4$, even considered over $\mathbb{C}$.
By \cite[Theorem~1.9]{Cha:2016-CG-bounds}, $C_R =10^9$ is an upper bound for the Cheeger-Gromov constant $C(M_R)$ of  the 0-surgery on $R$. For $k=1, \dots, N$,
let $J_k$ be a knot with $\Arf(J_k)=0$ and
\begin{align} \label{eqn:choosingjknots}
  \int_{\omega \in S^1} \sigma_{\omega}(J_k) \,d \omega > 2(2g+N-1)+ C_RN +
  \sum_{i=1}^{k-1}\Big( \int_{\omega \in S^1} \sigma_{\omega}(J_i) \, d \omega\Big).
\end{align}
We can achieve this by taking $J_k$ to be a sufficiently large even connected
sum of negative trefoils, since for the negative trefoil
\[
  \int_{\omega \in S^1} \sigma_{\omega}(J_i) \, d \omega = 4/3>0.
\]
In fact, the numerically minded reader can easily verify that
Equation~(\ref{eqn:choosingjknots}) is satisfied if we define $J_k$ to be the
connected sum of $10^{k+10} g$ negative trefoils.

We note that $K=\#_{i=1}^N R_\alpha(J_i)$ is a 2-solvable knot (by
Proposition~\ref{prop:infection2solvable}) which satisfies the hypotheses of
Theorem~\ref{thm:general}, and hence has topological 4-genus at least~$g+1$.

\subsection{Example 2: a 2-solvable, 2-bipolar knot with large 4-genus}
\label{subsection:example2bipolar}

Let $R$ be the knot depicted  on the left of Figure~\ref{fig:11n74}. On the
right of Figure~\ref{fig:11n74} we see a genus 2 Seifert surface $F$ for $R$
along with two sets of \emph{derivative curves} for $R$: each of the two component links
 $D=D_1 \cup D_2$ (blue) and $L= L_1 \cup L_2$ (red)
  generates a half-rank summand of $H_1(F)$,  forms a
slice link (in fact, an unlink) in $S^3$, and is 0-framed by~$F$.

\begin{figure}[ht]
  \includegraphics[height=6cm]{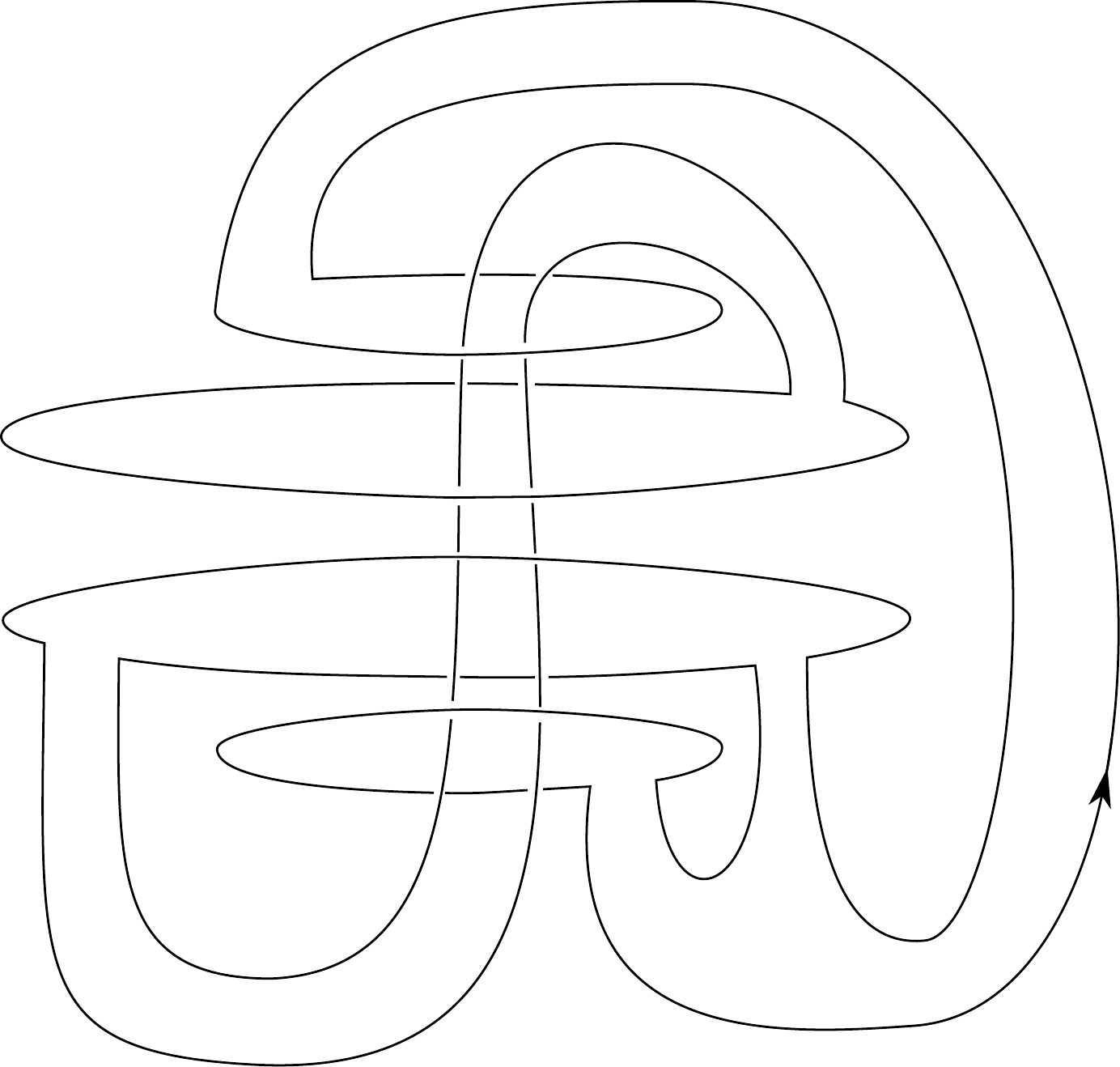} \qquad \qquad
  \includegraphics[height=6cm]{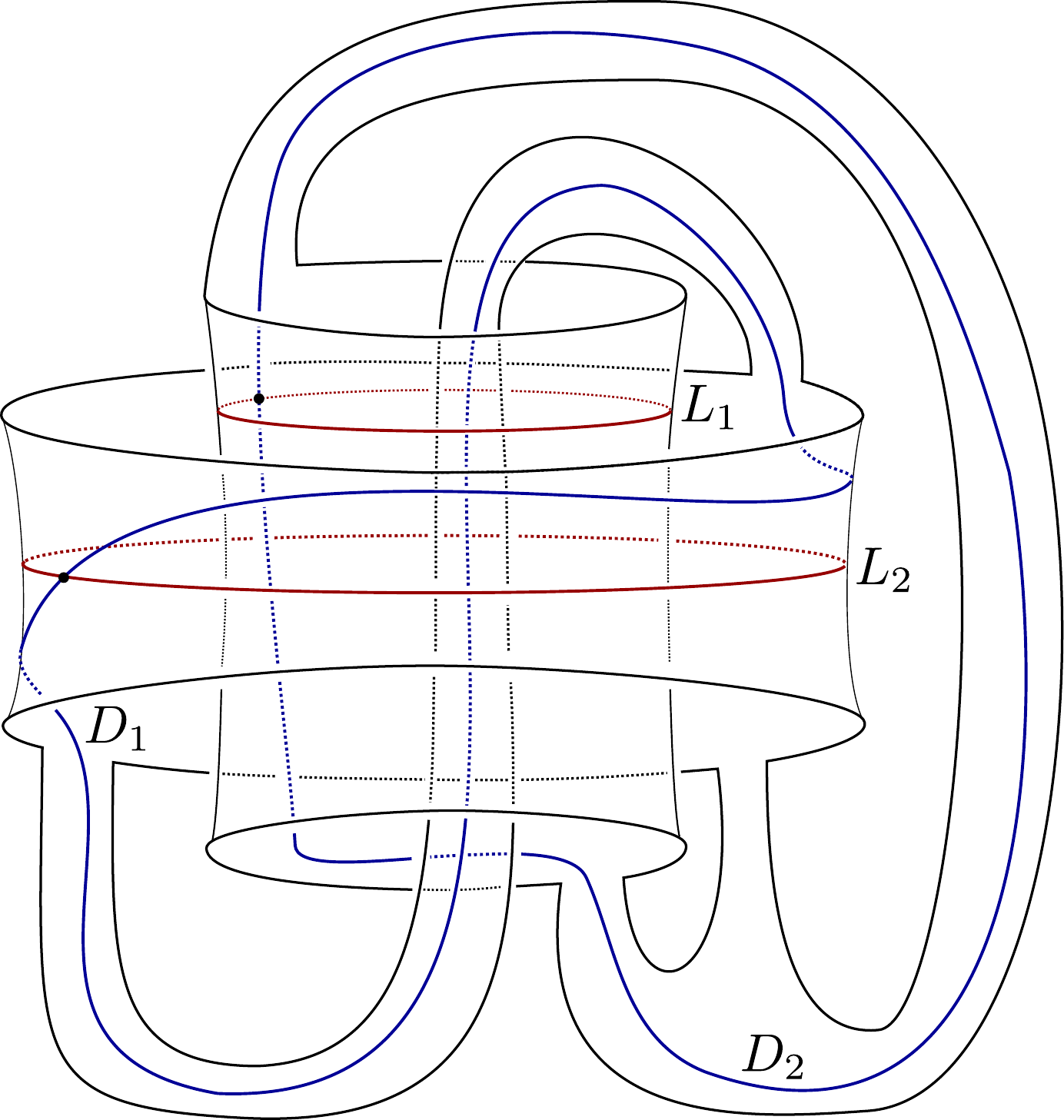}
  \caption{The knot $R=11_{n74}= P(3,-3,3,-2)$ (left) and a Seifert surface for
    $R$ along with two derivative links $L=L_1 \cup L_2$ and $D= D_1 \cup D_2$(right).}
  \label{fig:11n74}
\end{figure}

Now let $\beta_1, \beta_2 ,\gamma_1, \gamma_2$ be the curves indicated on the
left of Figure~\ref{fig:11n742}. These curves are disjoint from $F$ and hence
lie in  $\pi_1(M_R)^{(1)}$. Note that the indicated basepoints should be thought
of as living in a plane `far below' the plane of the diagram; in that plane
they can be connected, uniquely up to homotopy, to a single preferred basepoint
for~$\pi_1(M_R)$.

\begin{figure}[ht]
\includegraphics[height=6cm]{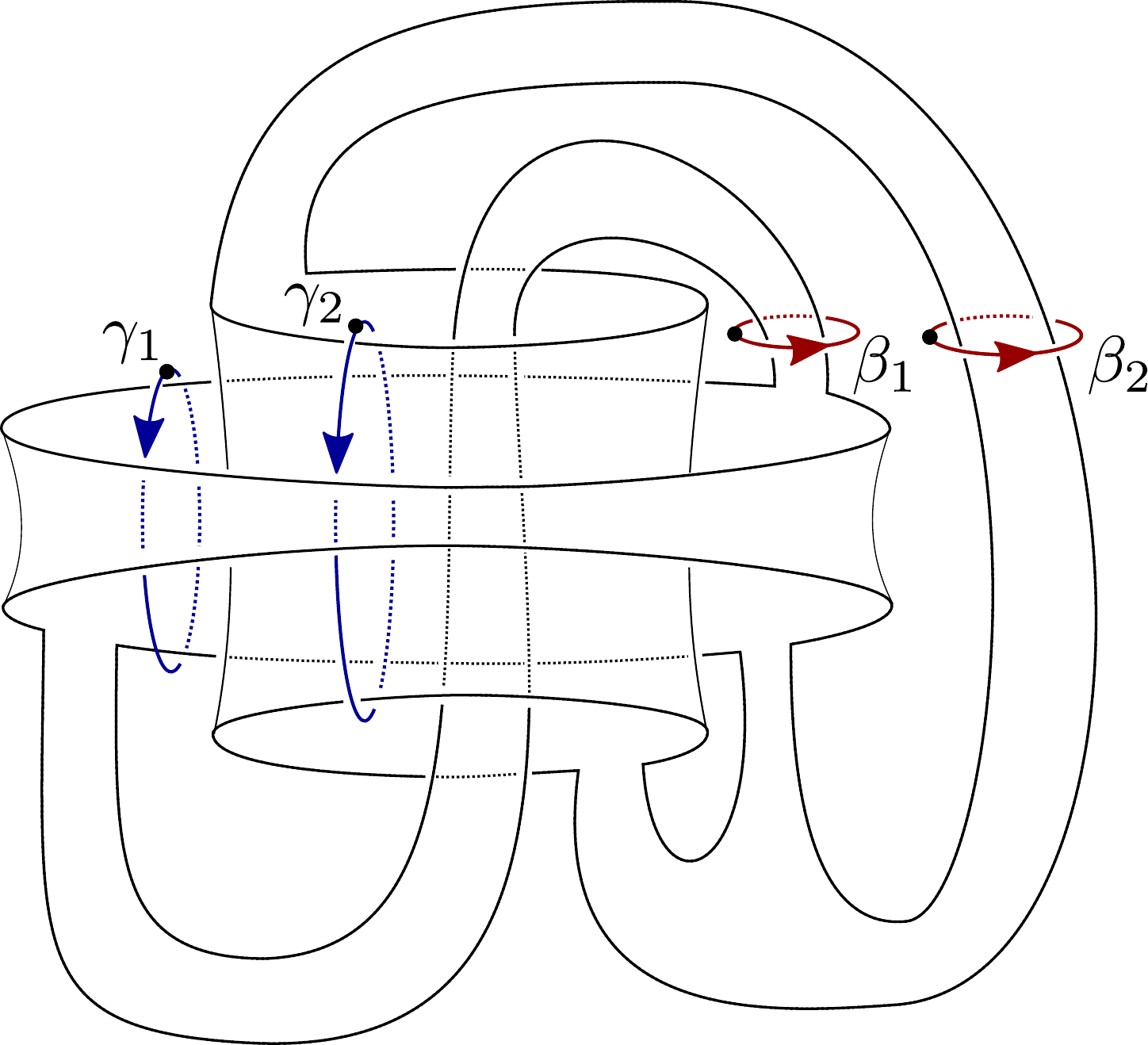}\qquad \quad
\includegraphics[height=6cm]{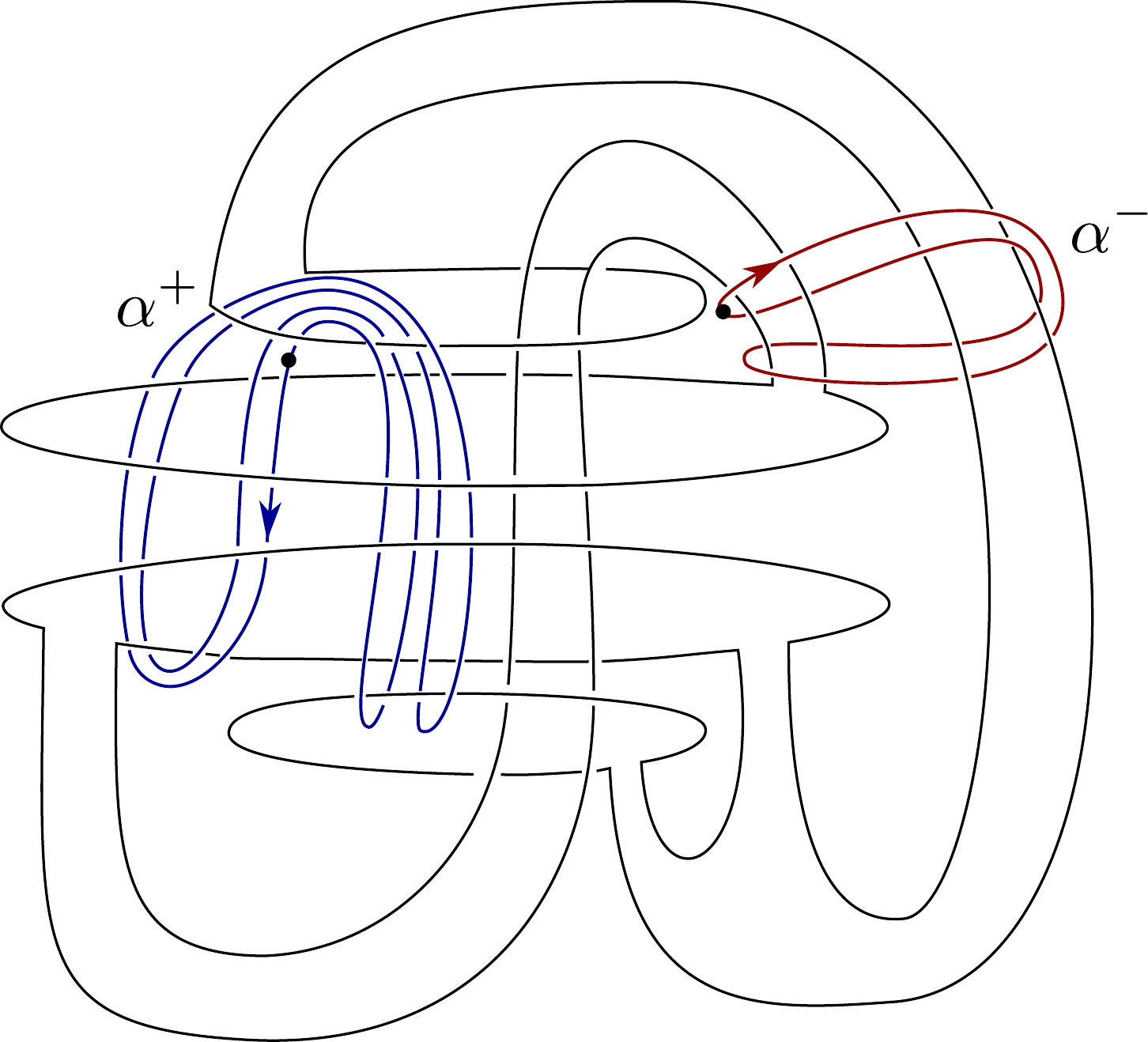}
\caption{Unknotted curves $\beta_1, \beta_2 ,\gamma_1, \gamma_2$ in
  $\pi_1(M_R)^{(1)}$ (left) and $\alpha^{\pm}$ in $\pi_1(M_R)^{(2)}$~(right).}
\label{fig:11n742}
\end{figure}

Let $\alpha^-=[\beta_1, \beta_2]$ and $\alpha^+=[\gamma_2, \gamma_2]$, where
$[v,w]=vwv^{-1}w^{-1}$; unknotted representatives for $\alpha^{\pm}$ are shown
on the right of Figure~\ref{fig:11n742}. Note that $\alpha^{\pm} \in
\pi_1(M_R)^{(2)}$. Since $\alpha^+$ has no geometric linking with either component of
the link $D$, for any knot $J$ the satellite knot
$R_{\alpha^+}(J)$ still has a smoothly slice derivative, and hence is itself
smoothly slice. Similarly, since $\alpha^-$ has no geometric linking with the
either component of the link $L$, the satellite knot $R_{\alpha^-}(J)$ is slice for every knot $J$.
Therefore, by Proposition~\ref{prop:buildingbipolar}, for any 0-positive knot
$J^+$ and 0-negative knot $J^-$, the knot $R_{\alpha^+, \alpha^-}(J^+, J^-)$ is
2-bipolar.

We now proceed to verify the conditions of Theorem~\ref{thm:general}, so that
for appropriate choices of $N \in \mathbb{N}$ and of $J^{\pm}_i$, $i=1, \dots,
N$, the knot $\#_{i=1}^N R_{\alpha^+, \alpha^-}(J_i^+, J_i^-)$ will have large
topological 4-genus.

Note that $H_1(\Sigma_2(R)) \cong \Z_3 \oplus \Z_3$, and so (up to rescaling by
a nonzero constant), there are four nontrivial characters $\chi\colon
H_1(\Sigma_2(R)) \to \Z_3$. We compute that for three of these characters, which
we call $\chi_1, \chi_2,$ and $\chi_3$, the resulting twisted homology is
\[
  H_1^{\theta \circ f_{\chi_i}}(M_R)
  \cong \Q(\xi_3)[t^{\pm1}]/ \langle(t-1)^2 \rangle=:A_1.
\]
For the fourth character, denoted by $\chi_4$,  we compute that the twisted
homology is
\[
  H_1^{\theta \circ f_{\chi_4}}(M_R) \cong \Q(\xi_3)[t^{\pm1}]
  / \langle t^2-14t+1 \rangle=:A_2.
\]
Crucially, for any nontrivial $\chi_i$ the lifts of $\alpha^+$ and $\alpha^-$ to
the cover of $M_R$ induced by $f_{\chi_i}$ generate $H_1^{\theta \circ
f_{\chi_i}}(M_R)$. (More precisely, $[1,0] \otimes [\alpha^+]$ and $[1,0]
\otimes [\alpha^-]$ generate.)

As in \cite{MilPow17}, we computed the twisted homology using a Maple program,
available for download on the authors' websites.  The program obtains a
presentation for the twisted homology using the Wirtinger presentation, taking
the Fox derivatives, and then applying the representation. It then simplifies
the presentation by row and column operations to obtain a diagonal matrix.
Keeping track of how the original generators, which can be identified in the
knot diagram, are modified under the sequence of row and column operations, we
not only compute the twisted homology $H_1^{\theta \circ f_{\chi_i}}(M_R)$ but
also can identify which elements the curves $[1,0] \otimes [\alpha^{\pm}]$
represent in $H_1^{\theta \circ f_{\chi_i}}(M_R)$. Note also that the orders of
$A_1$ and $A_2$ are both relatively prime to $\Delta_R(t)= (t^2-t+1)^2$ even over $\mathbb{C}$.

Now let $g \in \mathbb{N}$ be given. By our discussion above, we have $m_R=d_R=2$, and so we let
\[N:=\frac{4g(2+1)+2}{2}= 6g+1.\]
Let $C_R$ denote an
upper bound for the Cheeger-Gromov constant $C(M_R)$. For each $i=1, \dots, N$,
successively pick $m_i$  to be  even and large enough that $J_i^+:= \#^{m_i}
T_{2,3}$  satisfies
\begin{align*}
  \rho_0(J_i^+)&<-2(2g+N-1)- N C_R
  + \sum_{j=1}^{i-1}  \rho_0(J_i^+)- \sum_{j=1}^{i-1} \rho_0(J_i^-),
\end{align*}
and then pick $m_i'$ to be even and large enough that $J_i^-:=-\#^{m_i'}T_{2,3}$
satisfies
\begin{align*}
  \rho_0(J_i^-)&> 2(2g+N-1)+ NC_R-\sum_{j=1}^{i}  \rho_0(J_i^+)
  + \sum_{j=1}^{i-1} \rho_0(J_i^-) .
\end{align*}
Note that in particular $\rho_0(J_i^+)<0< \rho_0(J_i^-)$ for all~$i$.

We now let  $K_i= R_{\alpha^+, \alpha^-}(J_i^+, J_i^-)$ and  $K= \#_{i=1}^N
K_i$.  Observe that $K$ is 2-solvable by
Proposition~\ref{prop:infection2solvable} and 2-bipolar by
Proposition~\ref{prop:buildingbipolar}; $K$ also satisfies the hypotheses of
Theorem~\ref{thm:general}, so $g_4(K) \geq g+1$.  In
Section~\ref{section:height-four-gropes}, we will show that $K$ bounds an
embedded grope of height four in the 4-ball.  The knot $K$ therefore gives the
example claimed in Theorem~\ref{thm:mainthm-intro}.

It now remains to prove Theorem~\ref{thm:general}.

\section{Controlling the size of some homology groups}
\label{section:controlling-homology-groups}

This section contains some technical results needed for the proof of
Theorem~\ref{thm:general}, with the theme that we need to control the size of
certain homology groups of some covering spaces.

We start this section with an elementary algebraic lemma. This lemma and the one
after it are very similar to, and are inspired by, results of Levine in
\cite{Levine:1994-1}, in particular Lemma 4.3 of Part I and Proposition~3.2 of
Part II.  To avoid citing lemmas that were written for a different situation,
and for the edification of the reader, we provide short self-contained proofs.

\begin{lem}\label{lemma:algebraic-fact-square-matrices}
  Let $F \colon M \to M$ be an endomorphism of a finitely generated free
  $\Z[\Z_2]$-module $M$ such that
  \[
    \Id \otimes F \colon \Z \otimes_{\Z[\Z_2]} M \to \Z \otimes_{\Z[\Z_2]} M
  \]
  is an isomorphism, where $\Z$ is a $\Z[\Z_2]$-module via the trivial action of $\Z_2$ on $\Z$.
  Then
  \[
    \Id \otimes F \colon \Q[\Z_2] \otimes_{\Z[\Z_2]} M
    \to \Q[\Z_2] \otimes_{\Z[\Z_2]} M
  \]
  is also an isomorphism.
\end{lem}
\begin{proof}
  Let $A+BT \in \Z[\Z_2]$ be the determinant of $F$, where $A,B \in \Z$ and $T
  \in \Z_2$ denotes the generator.  Then $A+B = \pm 1$ by hypothesis. Thus
  $A^2-B^2 = (A+B)(A-B) = \pm (A-B) \neq 0$ since $A+B \equiv A-B$ modulo~$2$.
  Now \[(A+BT)\cdot \frac{1}{A^2-B^2}(A-BT) = 1,\] so over $\Q[\Z_2]$ the
  determinant of $F$ is invertible, and hence $\Id_{\Q[\Z_2]} \otimes F$ is an
  isomorphism as desired.
\end{proof}

Next we apply this lemma to obtain some control on the size of the homology of
double covering spaces.

\begin{lem}\label{lem:chain-homotopy-lifting}
  Let $f \colon X \to Y$ be a map of finite CW complexes such that
  \[
    f_* \colon H_i(X;\Z) \to H_i(Y;\Z)
  \]
  is an isomorphism for $i=0$ and a surjection for $i=1$.  Let $\varphi \colon
  \pi_1(Y) \to \Z_2$ be a surjective homomorphism and let $X^2,Y^2$ be the
  induced 2-fold covers.  Then
  \[
    f_* \colon H_i(X^2;\Q) \cong H_i(X;\Q[\Z_2])
    \to H_i(Y^2;\Q) \cong H_i(Y;\Q[\Z_2])
  \]
  is also an isomorphism for $i=0$ and a surjection for $i=1$.
\end{lem}

\begin{proof}
  The zeroth and first relative homology groups vanish, that is $H_i(Y,X;\Z)=0$
  for $i=0,1$.  Thus the cellular chain complex $(C_*(Y,X;\Z),\partial_*)$
  admits a partial chain contraction: writing $C_*$ to abbreviate $C_*(Y,X;\Z)$,
  the partial chain homotopy comprises maps $s_0 \colon C_0 \to C_1$ and $s_1
  \colon C_1 \to C_2$ such that $\partial \circ s_0 = \Id \colon C_0 \to C_0$
  and $\partial \circ s_1 + s_0 \circ \partial = \Id \colon C_1 \to C_1$.

  To see this, follow the proof of the fundamental lemma of homological algebra:
  for each basis element $x_i \in C_0$, choose a lift $y_i \in C_1$ with
  $\partial y_i = x_i$, and define $s_0(x_i) =y_i$, and then extend linearly.
  Such a $y_i$ exists since $\partial \colon C_1 \to C_0$ is surjective.  Then
  for each generator $z_i \in C_1$, consider $\Id(z_i) - s_0 \circ
  \partial(z_i)$. Since
  \[
    \partial (\Id(z_i) -s_0\circ \partial(z_i)) = \partial (z_i)
    - \partial \circ s_0 \circ \partial (z_i) = \partial (z_i) - \Id \circ
    \partial (z_i) =0,
  \]
  we have that $\Id(z_i) -s_0\circ \partial(z_i)$ is a cycle.  Hence there is a
  $v_i \in C_2$ such that $\partial v_i = \Id(z_i) -s_0\circ \partial(z_i)$.
  Define $s_1(z_i):= v_i$, and extend linearly to define $s_1$ on all of $C_1$.
  Then $\partial \circ s_1 (z_i)=  \partial v_i = \Id(z_i) -s_0\circ
  \partial(z_i)$ for every generator $z_i$ of $C_1$.  This completes the
  construction of a partial chain homotopy.

  Now consider the chain complex $D_*:= C_*(Y,X;\Z[\Z_2]) \cong C_*(Y^2,X^2)$,
  the relative chain complex of the 2-fold covering spaces.  Since the cellular
  chain groups are finitely generated free modules, the partial chain
  contraction $s_0,s_1$ lifts to maps $\wt{s}_0 \colon D_0 \to D_1$ and
  $\wt{s}_1 \colon D_1 \to D_2$.

  We claim that these maps induce a partial chain contraction after tensoring
  over $\Q[\Z_2]$.  To see the claim, the maps
  \[
    F:= \wt{s}_0 \circ \partial \colon D_0 \to D_0 \quad\text{and}\quad
    G:= \partial \circ \wt{s}_1 + \wt{s}_0 \circ \partial \colon D_1 \to D_1
  \]
  are endomorphisms of the free modules $D_0$ and $D_1$ respectively, that
  become automorphisms when tensored over~$\Z$.  That is,
  \begin{align*}
    \Id \otimes F \colon \Z \otimes_{\Z[\Z_2]} D_0 &\to \Z \otimes_{\Z[\Z_2]} D_0,
    \\
    \Id \otimes G \colon \Z \otimes_{\Z[\Z_2]} D_1 &\to \Z \otimes_{\Z[\Z_2]} D_1
  \end{align*}
  are isomorphisms.  By Lemma~\ref{lemma:algebraic-fact-square-matrices},
  \begin{align*}
    \Id \otimes F \colon \Q[\Z_2] \otimes_{\Z[\Z_2]} D_0
    & \to \Q[\Z_2] \otimes_{\Z[\Z_2]} D_0,
    \\
    \Id \otimes G \colon \Q[\Z_2] \otimes_{\Z[\Z_2]} D_1
    & \to \Q[\Z_2] \otimes_{\Z[\Z_2]} D_1
  \end{align*}
  are also isomorphisms.  Thus \[\Id \otimes \wt{s}_i \colon \Q[\Z_2]
  \otimes_{\Z[\Z_2]} D_i \to \Q[\Z_2] \otimes_{\Z[\Z_2]} D_{i+1}\] is a partial
  chain contraction and
  \[
    H_i(\Q[\Z_2] \otimes_{\Z[\Z_2]} D_*)
    = H_i(Y^2,X^2;\Q) \cong H_i(Y,X;\Q[\Z_2])=0
  \]
  for $i=0,1$.  The lemma follows from the long exact sequence of the pair
  ~$(Y,X)$ (Proposition~\ref{prop:les-pairs}).
\end{proof}

Our next lemma requires some facts about finitely generated modules over
commutative PIDs, which we remind the reader of in order to establish notation.

\begin{defn}
  Let $S$ be a commutative PID with quotient field $Q$, and let $A$ be a
  finitely generated module over~$S$.

  \begin{enumerate}

    \item $TA:=\{ a \in A \text{ such that } sa= 0 \text{ for some } s \in S\}$,
    the $S$-torsion submodule of~$A$.

    \item  $A^\wedge:= \Hom_S(A, Q/S)$.  If $A$ is torsion (i.e.\ $A=TA$), then $A$
    and $A^\wedge$ are non-canonically isomorphic.

    \item  Given a map of $S$ modules $f\colon A \to B$,  we abbreviate
    $f|_{TA}\colon TA \to TB$ by $f|_T$. We emphasize that $\coker(f|_T)$ is
    therefore isomorphic to $TB/ \Imm(f|_{TA})$, not $B/ \Imm(f|_{TA})$.

    \item We say that $A$ has \emph{generating rank  $k$ over $S$} if $A$ is
    generated as an $S$-module by $k$ elements but not by $k-1$ elements, and
    write $\grr_S A=k$. It follows immediately from the definition that if $A$
    surjects onto $B$ then $\grr_S B \leq \grr_S A$. It is also true and easy to
    check that if $B \leq A$ then $\grr_S B \leq \grr_S A$, though this is less
    obvious and uses that $S$ is a commutative~PID.

    \item By the fundamental theorem of finitely generated modules over PIDs,
    there exist $j, k \in \mathbb{N}$ and  elements $s_1, \dots, s_k \in S$ such
    that there is a (non-canonical) isomorphism \[A \cong S^j \oplus TA \cong
    S^j \oplus \bigoplus_{i=1}^k S/ \langle s_i \rangle.\] When $j>0$ we say
    that the \emph{order} of $A$ is $|A|=0$ and when $j=0$ we say that the order
    of $A$ is $|A|= \prod_{i=1}^k s_i$. This is well-defined up to
    multiplication by units in $S$. The key property of order we use is that if
    $f\colon A \to B$ is a map of $S$-modules with $\ker(f)$ torsion, then
    $|\Imm(f)|= |A|/|\ker(f)|$.

  \end{enumerate}
\end{defn}

We will need the following basic lemma in the proof of Theorem~\ref{thm:general}, noting for future use that $\mathbb{F}[t^{\pm1}]$ is a Euclidean domain whenever $\mathbb{F}$ is a field.

\begin{lem}\label{lem:intersecttorsion}
Let $A$ be a finitely generated module over a Euclidean domain $S$, hence non-canonically isomorphic to $S^{m} \oplus TA$ for some $m \geq 0$.
Suppose that $B$ is a submodule of $A$ such there exists $C \subseteq A/B$ of generating rank $\ell>m$.
Then there exists a module  $C' \subseteq TA/ (B \cap TA)$ of generating rank $\ell-m$ such that the order of $C'$ divides the order of $C$.
\end{lem}
\begin{proof}
Let $a_1, \dots, a_{\ell}$ be elements of $A$ such that $[a_1], \dots, [a_\ell]$ generate $C \subseteq A/B$. Pick a decomposition of $A\cong S^m \oplus TA$  and use it to write each $a_i= (s_i^j)_{j=1}^m \oplus \alpha_i$ for $(s_i^j)_{j=1}^m \in S^m$ and $\alpha_i \in TA$.
Since $S$ is a Euclidean domain, row-reduction of the $\ell \times m$ matrix $M$ with $M_{i,j}:=s_{i}^j$ yields a matrix $M'$ in Hermite normal  form. Since $\ell>m$, we have that $M'$ contains at least $\ell-m$ rows of zeros.
By taking the corresponding linear combinations of the $a_i$, we obtain a new collection of elements $a'_i= (t_i^j)_{j=1}^m \oplus \alpha'_i$ such that the collection of $[a'_i]$ generate $C$. Moreover, for $i>m$ we have that $t_i^j=0$ for all $j=1, \dots, m$ and so $a'_i= \alpha'_i \in TA$. Note that for $i>m$ we have that $\alpha_i' \neq 0$, since if $\alpha_i=a'_i=0$ then the generating rank of $C$ would be strictly less than $\ell$. For similar reasons, we see that the generating rank of the submodule of $TA/ (B \cap TA)$ generated by $\alpha'_{m+1}, \dots, \alpha'_{\ell}$ is exactly $\ell-m$. So let $C'$ be this submodule. Finally, the order of $C'$ divides the order of $C$ because $C'$ is a submodule of $C$.
\end{proof}

We will use the next lemma twice, once in the proof of
Proposition~\ref{prop:homology}, and then again in Step 3 of the proof of
Theorem~\ref{thm:general} in Section~\ref{section:proof-of-main-theorem}.

\begin{lem}\label{lemma:moregeneralhomology}
  Let $X$ be a 4-manifold with boundary $\partial X=Y$. Let $S$ be a commutative
  PID with quotient field $Q$ and let $A := S^n$ for some $n \in \mathbb{N}$.
  Suppose there is a representation $\Phi$ of the fundamental groupoid of $Y$
  into $\Aut(A)$ that extends over~$X$, as in
  Section~\ref{section:disconnected}. Consider the long exact sequence
  (Proposition~\ref{prop:les-pairs}) of $S$-modules of the pair, with
  coefficients taken in~$A$:
  \begin{align*}
    \cdots \to H_2(X) \xrightarrow{j_2} H_2(X,Y) \xrightarrow{\partial} H_1(Y) \xrightarrow{i_1} H_1(X) \xrightarrow{j_1} H_1(X,Y)  \to \cdots.
  \end{align*}
  Suppose that $H_1(X,Y)$ is torsion. Then $\ker(j_1|_T) \cong \coker(j_2|_T)$.
\end{lem}

\begin{proof}
  Unless otherwise specified, all homology groups are taken with coefficients in
  $A$. First, we argue that the Bockstein homomorphism $\beta$ is an isomorphism.
  This Bockstein arises in the long exact sequence of $\Ext$
  groups~\cite[IV,~Prop.~7.5]{HS97} associated to the short exact sequence $0
  \to S \to Q \to Q/S \to 0$, as follows:
  \[
    \Ext^0_S(H_1(X,Y), Q) \to \Ext^0_S(H_1(X,Y), Q/S)
    \xrightarrow{\beta} \Ext^1_S(H_1(X,Y),S)\to \Ext^1_S(H_1(X,Y), Q).
  \]
  Since $Q$ is an injective $S$-module, $\Ext^1_S(H_1(X,Y),Q)$ vanishes, and
  $\Ext^0_S(H_1(X,Y),Q)=0$ because $H_1(X,Y)$ is torsion.  Thus $\beta$ is an
  isomorphism.

  Therefore Poincar{\'e} duality, universal coefficients, and the
  (inverse of the) Bockstein homomorphism together induce natural isomorphisms
  fitting into a commutative diagram:
  \[
    \xymatrix@R+.3cm@C-2.5ex{
      TH_2(X) \ar[r]^-{P.D.}_-{\cong} \ar[d]^{j_2|_T} & TH^2(X,Y)  \ar[r]^-{U.C.}_-{\cong} \ar[d] &\Ext^1_S(H_1(X,Y),S) \ar[r]^-{\beta^{-1}}_-{\cong} \ar[d] &\Ext^0_S(H_1(X,Y), Q/S) \ar[r]_-{\cong} \ar[d] & TH_1(X,Y)^\wedge \ar[d]_-{(j_1|_T)^\wedge} \\
      TH_2(X,Y) \ar[r]^-{P.D.}_-{\cong} & TH^2(X) \ar[r]^-{U.C.}_-{\cong} & \Ext^1_S(H_1(X), S)
      \ar[r]^-{\beta^{-1}}_-{\cong} & \Ext^0_S(H_1(X), Q/S) \ar[r]_-{\cong} & TH_1(X)^{\wedge}
    }
  \]

  By the naturality of the above sequence of maps, the following square commutes
  and so $ \coker(j_2|_T) \cong \coker((j_1|_T)^\wedge)$.
  \begin{equation}\label{eqn:squaregeneral}
    \vcenter{\xymatrix @R+0.3cm @C+0.3cm {
      TH_2(X) \ar[r]^{j_2|_T} \ar[d]_{\cong} & TH_2(X,Y) \ar[d]^{\cong}\\
      TH_1(X,Y)^{\wedge}  \ar[r]^-{(j_1|_T)^\wedge} & TH_1(X)^{\wedge}
    }}
  \end{equation}
  Now let $H := \ker(j_1|_T) \leq TH_1(X)$ and define $\Phi\colon
  \coker((j_1|_T)^\wedge) \to H^\wedge$ by $\Phi([f])= f|_{H}.$ Observe that
  $\Phi$ is well-defined, since for any $g \in TH_1(X,Y)^\wedge$ we have
  \[
    (j_1|_T)^\wedge(g)(x)= g(j_1(x))= g(0)=0
    \quad\text{for all } x \in H= \ker(j_1|_T).
  \]
  Also, $\Phi$ is onto: given any $f \in H^\wedge$ (i.e.\ a map $f\colon H \to
  Q/S$), since  $H \leq TH_1(X)$, and using that $Q/S$ is an injective
  $S$-module, we can extend $f$ to a map $g\colon TH_1(X) \to Q/ S$, and will
  have that $\Phi([g])=f$.  Therefore, in order to show that $\Phi$ is an
  isomorphism it is enough to show that $ |\coker((j_1|_T)^\wedge)|=
  |H^\wedge|$. Note that
  \begin{equation}\label{eqn:coker1general}
    | \coker((j_1|_T)^\wedge)|= \frac{|TH_1(X)^\wedge|}{|\Imm((j_1|_T)^\wedge)|}
    = \frac{ |TH_1(X)^\wedge|\,|\ker((j_1|_T)^\wedge)|}{|TH_1(X,Y)^\wedge)|}.
  \end{equation}
  Also, $(j_1|_T)^{\wedge}(f)=0$ if and only if $f(j_1|_T(x))=0$ for all $x \in
  TH_1(X)$ if and only if $f$ vanishes on $\Imm(j_1|_T)$, so
  \[
    |\ker((j_1|_T)^\wedge)|= \frac{|TH_1(X,Y)^\wedge)|}{|\Imm(j_1|_T)|}.
  \]
  Therefore we can rewrite Equation~(\ref{eqn:coker1general}) as
  \[
    |\coker((j_1|_T)^\wedge)|= \frac{|TH_1(X)^\wedge|}{|\Imm(j_1|_T)|}
    =\frac{|TH_1(X)|}{|\Imm(j_1|_T)|} = |H|= |H^{\wedge}|.
  \]
  So $\Phi\colon \coker((j_1|_T)^\wedge) \to \ker(j_1|_T)^\wedge$ is an
  isomorphism. Since $\ker(j_1|_T)$ is a torsion $S$-module, there is an (albeit
  non-canonical) identification $\ker(j_1|_T) \cong (\ker(j_1|_T))^\wedge$, and
  so we have as desired that
  \[
    \coker(j_2|_T) \cong \coker((j_1|_T)^\wedge)
    \cong_\Phi \ker(j_1|_T)^\wedge \cong \ker(j_1|_T).
    \qedhere
  \]
\end{proof}

Note that in particular this implies that $\ker(j_1|_T)_p  \cong
\coker(j_2|_T)_p$ for any prime $p$, where for a $\Z$-module $G$ we write $G_p$
for the $p$-primary part.

Next we apply the control on homology of double covers gained in Lemma~\ref{lem:chain-homotopy-lifting} along with the homological algebra of Lemma~\ref{lemma:moregeneralhomology}  to  manifolds $M^3$ and $V^4$.

\begin{prop}\label{prop:homology}
  Let $M$ be a homology $S^1 \times S^2$, and let $V$ be a connected 4-manifold with
  boundary $M$ such that the inclusion induced map $H_1(M) \xrightarrow{i_*}
  H_1(V)$ an isomorphism. Suppose that $H_1(M^2) \cong \Z \oplus G$, where $G$
  is torsion. Then for any prime $p$ the $p$-primary part of $TH_1(M^2)/
  \ker(TH_1(M^2) \to TH_1(V^2))$ has generating rank at least
  $n:=\frac{m-2\chi(V)}{2}$, where $m$ denotes the generating rank of the
  $p$-primary part of $TH_1(M^2)$.
\end{prop}

\begin{proof}
  Observe that $H_i(V, M; \Z)=0$ for $i=0, 1$. It follows from
  Lemma~\ref{lem:chain-homotopy-lifting} that
  \[
    H_i(V^2, M^2; \Q) \cong H_i(V, M; \Q[\Z_2]) =0 \quad\text{for } i=0,1.
  \]
  Therefore $\dim H_1(V^2, \Q) \leq \dim H_1(M^2, \Q) =1$ and
  \[
    \dim H_3(V^2; \Q)= \dim H^3(V^2; \Q)= \dim H_1(V^2, M^2; \Q)=0.
  \]
  Also note that
  \[
    2\chi(V)=\chi(V^2)= 1 - \dim H_1(V^2, \Q)+ \dim H_2(V^2, \Q).
  \]
  We therefore have that $\dim H_2(V^2, \Q) = 2\chi(V) + \dim H_1(V^2, \Q) - 1$
  is at most~$2\chi(V)$. Now consider the following long exact sequence:
  \[
    \dots \to H_2(V^2) \xrightarrow{j_2} H_2(V^2, M^2)
    \xrightarrow{\partial} H_1(M^2) \xrightarrow{i_1} H_1(V^2)
    \xrightarrow{j_1} H_1(V^2, M^2) \dots.
  \]
From above we have $H_1(V^2, M^2; \Q)=0$, and so by
  Lemma~\ref{lemma:moregeneralhomology} we have that $\grr \ker(j_1|_T)_p= \grr
  \coker(j_2|_T)_p$.  Moreover, for any finitely generated abelian group $A$, we have that
  $A \cong \Z^{a} \oplus \, TA$ for some $a \in \mathbb{N}_{\geq 0}$ and hence
  that $\grr(A \otimes \Z_p)= a + \grr(A_p)$.  In particular $\grr(\ker(j_1|_T) \otimes
  \Z_p)= \grr (\coker(j_2|_T) \otimes \Z_p)$.
  Combining  $\grr(A  \otimes \Z_p)= a + \grr(A_p)$  with $H_1(M^2) \cong \Z \oplus
  TH_1(M^2)$,  we  have that in order to show as desired that the generating rank of the
  $p$-primary part of $\left(TH_1(M^2)/ \ker(i_1|_T)\right)$ is at least $n$, it
  suffices to show
  \[
    \grr \left( (H_1(M^2)/ \ker(i_1)) \otimes \Z_p \right) \geq n+1.
  \]

  Note that $H_1(M^2)/ \ker(i_1)\cong \Imm(i_1) \cong \ker(j_1)$, so if
  $\grr\left(\ker(j_1) \otimes \Z_p\right) \geq n+1$ then we are done.
  Similarly, since $\grr\left(H_1(M^2) \otimes \Z_p\right) = m+1$ and
  \[
    \ker(i_1)= \Imm(\partial) \cong H_2(V^2, M^2)/ \ker(\partial)
    = H_2(V^2, M^2)/ \Imm(j_2)= \coker(j_2),
  \]
  if $\grr \left(\coker(j_2) \otimes \Z_p\right) \leq m-n$, then we are also
  done.

  So suppose for a contradiction $\grr\left( \coker(j_2) \otimes \Z_p\right)>
  m-n$ and $\grr\left( \ker(j_1)\otimes \Z_p \right) \leq n$.  Note that since
  $H_3(V^2,M^2) \cong H^1(V^2) \cong H^1(M^2) \cong H_2(M^2)$, the ranks of
  $H_2(V^2,M^2)$ and $H_2(V^2)$ coincide and so $H_2(V^2,M^2)$ splits (albeit
  non-canonically) as $\Z^{b_2(V^2)} \oplus TH_2(V^2,M^2)$. Thus we obtain our
  desired contradiction as follows:
  \begin{align*}
    2\chi(V)+n=m-n&< \grr \left( \coker(j_2) \otimes \Z_p \right) \\
    &\leq \grr \left((H_2(V^2,M^2)/ \Imm(j_2|_{T})) \otimes \Z_p \right)\\
    &= \grr \left(\Z^{b_2(V^2)} \otimes \Z_p \right)
    + \grr \left( (TH_2(V^2, M^2)/ \Imm(j_2|_T))  \otimes \Z_p\right) \\
    &= b_2(V^2) + \grr \left(\coker(j_2|_T )\otimes \Z_p\right) \\
    &= b_2(V^2)+ \grr \left(\ker(j_1|_T) \otimes \Z_p\right)\\
    &\leq b_2(V^2)+ \grr \left(\ker(j_1) \otimes \Z_p\right)\leq 2\chi(V)+n.
    \qedhere
  \end{align*}
\end{proof}

\section{A standard cobordism}\label{section:a-standard-cobordism}

In this section we study a standard cobordism $U$ between the zero-framed
surgery manifold of a connected sum of knots $M_K=M_{\#_{i=1}^N K_i}$ and the
disjoint union $Y:=\bigsqcup_{i=1}^N M_{K_i}$ of the zero-framed surgery
manifolds of the summands $K_i$. In particular we need to understand the
behavior of certain representations of the fundamental groups. We will also
explicitly choose the basepoints $\{x_i\}$ and paths $\{\tau_i\}$ necessary to
define twisted homology for disconnected manifolds, as discussed  in
Section~\ref{section:disconnected}.

\begin{figure}[htbp]
  \includegraphics[height=1.3in]{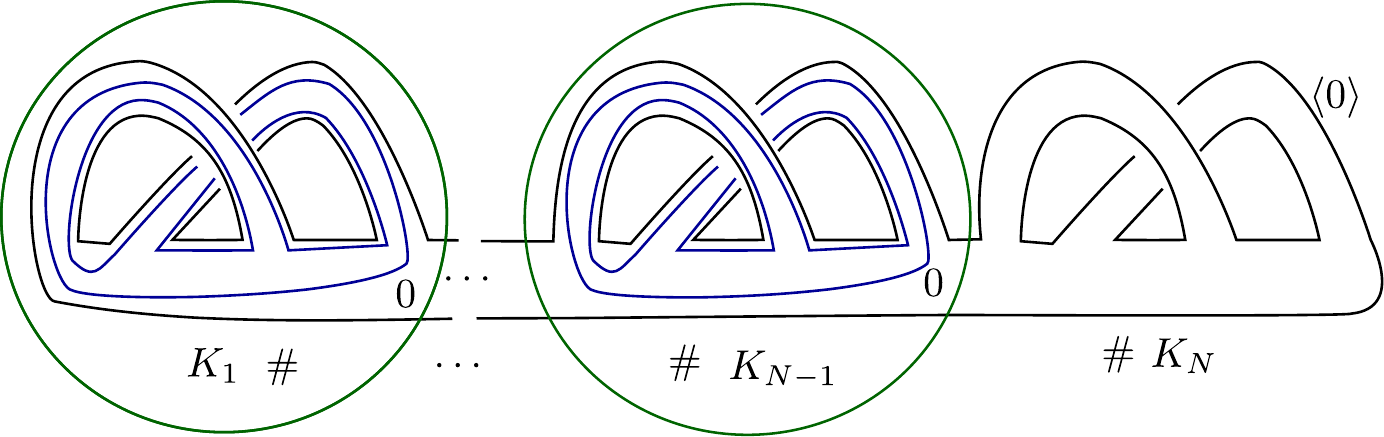}
  \caption{A Kirby diagram for $U$.}
  \label{fig:upart1}
\end{figure}

Let $U'$ be $M_K \times [0,1]$ with $(N-1)$ 0-framed 2-handles attached along
`longitudes of $K_i$.' A schematic of a relative Kirby diagram for $U'$ is given
by the black and blue curves of Figure~\ref{fig:upart1}. Note that we  depict
each $K_i$ as the boundary of a Seifert surface $G_i$, and hence $K=\#_{i=1}^N
K_i$ as the boundary of $\natural_{i=1}^N G_i$.  Since repeatedly sliding the
black 0-framed curve over the blue curves gives the standard surgery diagram for
$Y'$, we have $\partial_+(U')= Y':= \#_{i=1}^N M_{K_i}$. Now let $U''$ be  $Y'
\times [1,2]$  together with $(N-1)$ 3-handles attached along 2-spheres (whose
outline is indicated in green in Figure~\ref{fig:upart1}) so that $\partial_+
U''=Y= \bigsqcup_{i=1}^N M_{K_i}$. Let $U= U' \cup_{Y'} U''$.

We now consider the points, arcs, and closed curves shown in
Figure~\ref{fig:ualex}.

\begin{figure}[H]
  \includegraphics[height=1.75in]{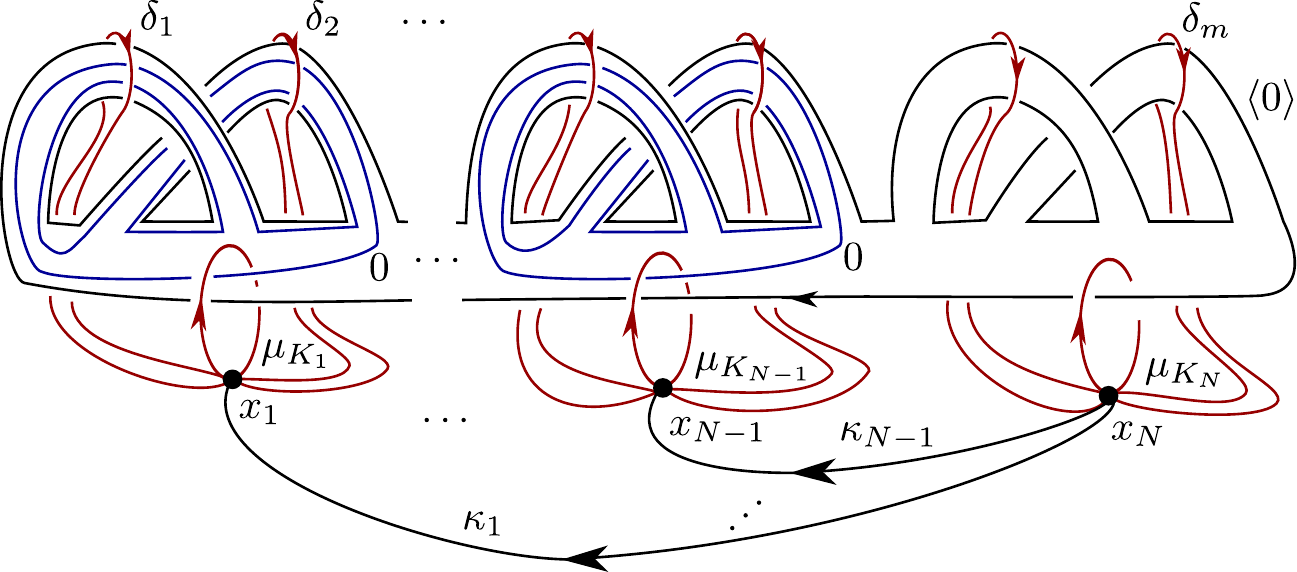}
  \caption{Basepoints $x_i$, arcs $\kappa_i$, and meridians $\mu_i$ for
    $i=1, \dots, N$ and closed curves $\delta_j$ for $j=1, \dots, m$ in $M_K$.}
  \label{fig:ualex}
\end{figure}

Note that the curves $\delta_j$ for $j=1, \dots, m$ form a normal generating set
for the first commutator subgroup of $\pi_1(M_K, x_N)$, when suitably based
using the arcs $\kappa_i$.

The attaching regions for the 2-handles of $U'$ avoid
\[
  \bigcup_{i=1}^N \mu_{K_i} \cup  \bigcup_{i=1}^{N-1} \kappa_i
  \cup \bigcup_{j=1}^m \delta_j \subset M_K,
\]
and so the points $x_i':= x_i \times \{1\}$,  arcs $\kappa_i':= \kappa_i \times
\{1\}$, and loops $\mu_{K_i}':= \mu_{K_i} \times \{1\}$  and $\delta_j':=
\delta_j \times \{1\}$ lie in $\partial_+U'=Y'$ for all $i=1, \dots, N$ and
$j=1, \dots, m$. Similarly, the attaching regions for the 3-handles of $U''$
avoid
\[
  \Big( \bigcup_{i=1}^N \mu_{K_i}' \cup \bigcup_{j=1}^m \delta_j' \Big)
  \subset Y',
\]
and so the loops $\mu_{K_i}'':= \mu_{K_i}' \times \{2\}$ and
$\delta_j'':=\delta_j' \times \{2\}$ lie in $\partial_+U''=Y$ for all $i=1,
\dots, N$ and $j=1, \dots, m$. For each $i=1, \dots N$ we have an
inclusion-induced map
\[
  \iota_i \colon \pi_1(M_{K_i}, x_i' \times \{2\}) \to \pi_1(U'', x_N')
  \quad\text{by }
  \beta\mapsto \kappa_i' \cdot (x_i' \times [1,2]) \cdot \beta \cdot
  \overline{(x_i' \times [1,2])} \cdot \overline{\kappa_i'}.
\]
Let $U=U' \cup_{Y'} U''$, and note that  we also have an inclusion-induced map
\[
  \iota\colon \pi_1(U'', x_N') \to \pi_1(U, x_N)
  \quad\text{by }
  \gamma \mapsto (x_N \times [0,1]) \cdot \gamma
  \cdot \overline{(x_N \times [0,1])}.
\]
In the language of Section~\ref{section:disconnected}, $\iota \circ \iota_i$ is
induced by the path from $x_N$ to $x_i''$ given by
\[
  \tau_i= (x_N \times [0,1]) \cdot \kappa_i' \cdot (x_i' \times [1,2]).
\]

We return to using the notation from Section~\ref{section:metabelian-homology} in order to state and prove the following.

\begin{prop}\label{prop:standardu}
  Let $K= \#_{i=1}^N K_i$ and  $U$ be the standard cobordism from $M_K$  to
  $\bigsqcup_{i=1}^N M_{K_i}$ as above. Let $p \in \mathbb{N}$ and choose maps
  $\chi_i\colon H_1(\Sigma_2(K_i)) \to \mathbb{Z}_p$ for $i=1$, \dots,~$N$, so
  $(\chi_i)_{i=1}^N\colon H_1(\Sigma_2(K)) \to \Z_p.$ Let $\mu_{K_N}$ be the
  preferred meridian for $K$ and  for  $i=1$, \dots,~$N$ let $\mu_{K_i}''$ be
  the preferred meridian for $K_i$. Then the map
  \[
    f^{K}_{(\chi_i)_{i=1}^N}\colon \pi_1(M_K, x_N) \to \Z \ltimes \Z_p
  \]
  extends uniquely to a map $F\colon \pi_1(U, x_N) \to \Z \ltimes \Z_p$. Also,
  the composition
  \[
    f_i\colon  \pi_1(M_{K_i}, x_i' \times \{2\}) \xrightarrow{\iota_i}
    \pi_1(U'', x_N') \xrightarrow{\iota} \pi_1(U, x_N) \xrightarrow{F}
    \Z \ltimes \Z_p
  \]
  satisfies $f_i= f^{K_i}_{\chi_i}$.
\end{prop}

\begin{proof}
  Notice that $\pi_1(U)= \pi_1(M_K)/ \langle \lambda_{K_1}, \dots,
  \lambda_{K_{N-1}} \rangle$. Therefore, since each $\lambda_{K_i}$ bounds a
  subsurface of $\natural_{i=1}^N G_i$ and hence lies in $ \pi_1(M_K)^{(2)}$,
  the map $f^K_{(\chi_i)_{i=1}^N}$ extends uniquely as desired.

  Observe that for all $i=1, \dots, N$ we have
  \begin{align*}
    \iota(\iota_i( \mu_{K_i}''))&= \iota( \kappa_i' \cdot (x_i' \times [1,2])
    \cdot \mu_{K_i}'' \cdot \overline{(x_i' \times [1,2])}
    \cdot \overline{\kappa_i'})
    \\
    & = \iota(\kappa_i' \cdot \mu_{K_i}'\cdot  \overline{\kappa_i'})
    \\
    &=(x_N \times [0,1]) \cdot \kappa_i'\cdot \mu_{K_i}'
    \cdot \overline{\kappa_i'} \cdot \overline{(x_N \times [0,1])}
    = \mu_{K_N} \in \pi_1(U, x_N).
  \end{align*}
  Therefore
  \[
    f_i(\mu_{K_i}'')= F(\mu_{K_N})= f^{K}_{(\chi_i)_{i=1}^N}(\mu_{K_N})
    = (t,0)= f_{\chi_i}^{K_i}(\mu_{K_i}'') \in \Z \ltimes \Z_p.
  \]

  For each $i=1, \dots N$, every element $\gamma \in \pi_1(M_{K_i}, x_i'')$ can
  be written as $\gamma= (\mu_{K_i}'')^{\varepsilon(\gamma_i)} a$ for some
  element $a \in \pi_1(M_{K_i}, x_i'')^{(1)}$. Moreover, the collection of
  $\delta_j''$ corresponding to $K_i$ in Figure~\ref{fig:ualex} normally
  generate $ \pi_1(M_{K_i}, x_i'')^{(1)}$.  It therefore suffices to check that
  $f_i$ agrees with $f^{K_i}_{\chi_i}$ on $\mu_{K_i}''$ (as done above) and on
  the collection of $\delta_j''$ corresponding to~$K_i$.

  Supposing that $\delta_j''$ corresponds to $K_{i}$, we have that

  \begin{align*}
    \iota(\iota_i( \delta_j''))&= \iota( \kappa_i' \cdot (x_i' \times [1,2])
    \cdot \delta_j'' \cdot\overline{(x_i' \times [1,2])} \cdot\overline{\kappa_i'})
    \\
    & = \iota(\kappa_i' \cdot \delta_j'\cdot  \overline{\kappa_i'})
    \\
    &=(x_N \times [0,1]) \cdot \kappa_i'\cdot \delta_j'
    \cdot \overline{\kappa_i'} \cdot \overline{(x_N \times [0,1])}
    \\
    &= \kappa_i \delta_j \overline{\kappa_i} \in \pi_1(U, x_N).
  \end{align*}

  Now fix a lift $\widetilde{x_N}$ of $x_N$ to $M_K^2$, the double cover of
  $M_K$.  Since $x_N$ does not lie in a tubular neighborhood of $K$, we can
  think of $\widetilde{x_N}$ as lying in $E_K^2\subseteq \Sigma_2(K)$ as well.
  The inclusion induced maps $M_K \to U$ and $M_{K_i} \to U$ induce isomorphisms
  on first homology, and so the double cover $U^2$ is a cobordism from $M_K^2$
  to $\bigsqcup_{i=1}^N M_{K_i}^2$. For each $i=1$, \dots, $N$, lifting the arc
  \[
    \kappa_i \cdot (x_i \times [0,1]) \cdot (x_i' \times [1,2])
  \]
  to $U^2$ starting at  $\widetilde{x_N}$ gives a preferred basepoint
  $\widetilde{x_i''}$ in $M_{K_i}^2$. As before, we also think of this basepoint
  as lying in $E_{K_i}^2 \subseteq \Sigma_2(K_i)$. We can therefore speak of
  \emph{the} lift $\widetilde{\gamma}$ of a curve $\gamma$ based at $x_N$
  (respectively, $x_i''$) to $\Sigma_2(K)$ (respectively,  $\Sigma_2(K_i)$) by
  choosing the lift with basepoint $\widetilde{x_N}$ (respectively,
  $\widetilde{x_i''}$).

  \begin{remark}
    A choice of basepoint is technically always necessary to define $f^K_\chi$,
    though this was suppressed in Section~\ref{section:metabelian-homology} in
    our discussion of the connected case, where it was less important.
  \end{remark}

  Therefore
  \begin{align*}
    f_i(\delta_j'') = F(  \kappa_i \delta_j \overline{\kappa_i})
    = f^{K}_{(\chi_k)_{k=1}^N}(\kappa_i \delta_j \overline{\kappa_i})
    =(0, (\chi_k)_{k=1}^N(\widetilde{\kappa_i \delta_j \overline{\kappa_i}})).
  \end{align*}
  Similarly,
  \begin{align*}
    f_{\chi_i}^{K_i}(\delta_j'')= (0, \chi_i(\widetilde{\delta_j''})).
  \end{align*}
  It therefore only remains to show that
  \[
    (\chi_k)_{k=1}^N(\widetilde{\kappa_i \delta_j \overline{\kappa_i}})=
    \chi_i(\widetilde{\delta_j''}).
  \]

  First, note that $\chi_k(\widetilde{\kappa_i \delta_j \overline{\kappa_i}})=0$
  unless $k=i$. Also, the homology class of $\widetilde{\kappa_i \delta_j
  \overline{\kappa_i}}$ in
  \[
    H_1(\Sigma_2(K_i)) \subset \bigoplus_{i=1}^N H_1(\Sigma_2(K_i))
    \cong H_1(\Sigma_2(K)
  \]
  is exactly the same as that of  $\widetilde{\delta_j''}$ in
  $H_1(\Sigma_2(K_i))$, and so we have that
  \[
    \chi_i(\widetilde{\kappa_i \delta_j \overline{\kappa_i}})
    = \chi_i(\widetilde{\delta_j''}).\qedhere
  \]
\end{proof}

We note for later use that the inclusion induced maps $M_K \to U$ and $M_{K_i}
\to U$ give isomorphisms on first homology, and that $H_2(U) \cong \Z^{N}$ and
$H_3(U) \cong \Z^N$.

\section{Proof of Theorem~\ref{thm:general}}\label{section:proof-of-main-theorem}
Since the proof of Theorem~\ref{thm:general} is rather long, for the reader's convenience we outline the main steps of the argument, with references to key results from elsewhere in the paper.

\begin{enumerate}
  \item (Proposition~\ref{prop:step1}.) Construct a 4-manifold $V$ with boundary $\partial V=M_K$ such that the inclusion
  induced map $H_1(M_K) \to H_1(V)$ is an isomorphism and $H_2(V) \cong \Z^{2g}$.
  Let $U$ denote the standard cobordism between $M_K$ and $Y:=
  \bigsqcup_{i=1}^{N}M_{K_i}$ discussed in Section~\ref{section:a-standard-cobordism}
   and let $Z:= V \cup_{M_K} U$. Note for later use that $H_2(Z)=\Z^{2g+N-1}$ and  $\chi(Z)=2g$.
  \item (Propositions~\ref{prop:standardu} and~\ref{prop:step2}.) Show that we can choose maps $\chi_i\colon H_1(M_{K_i}) \to \Z_p$ such
  that the corresponding map $\phi\colon \coprod_{i=1}^N \pi_1(M_{K_i}) \to \Z
  \ltimes \Z_p$ extends to a map $\Phi\colon \pi_1(Z) \to \Z \ltimes \Z_{p^a}$
  for some $a \geq 1$ and such that  at least $n:=\frac{m_RN-4g}{2}$ of the $\chi_i$  are
  nonzero.
  \item (Claim~\ref{prop:step-3}.) Show that for some $1 \leq i \leq N$ and $1 \leq j \leq r$,  the element $ [1,0] \otimes [\lambda(\eta_i^j)] $ in $H_1^\phi(Y) = H_1 \bigl(\Q(\xi_p)[t^{\pm1}]^2 \otimes_{\Z[\pi_1(Y)]} C_*(\widetilde{Y}) \bigr)$ does not map to $0$ in $H_1^{\Phi}(Z)$. (Recall that $\lambda(\eta_i^j)$ is a longitude of the infection curve $\eta_i^j$ and lies in $M_{K_i} \subset Y$.)  This step, which contains much of the technical work of the theorem,  crucially relies on our
  assumption that for every nontrivial $\chi\colon  H_1(\Sigma_2(R)) \to \Z_p$ we
  know that  the collection $\{[1,0] \otimes [\eta^j]\}$ generates $H_1^{\theta
  \circ f_\chi}(R)$ and that the order of $H_1^{\theta \circ f_\chi}(R)$ is
  relatively prime to $\Delta_R(t)$.
  \item (Last two paragraphs of Section~\ref{section:proof-of-main-theorem}.)  Construct a local coefficient derived series representation $\pi_1(Y)
  \to \Lambda$ extending over $\pi_1(Z)$ and bound the $L^{(2)}$
  $\rho$-invariant $\rho^{(2)}(Y,\Lambda)$ in two different ways to get a
  contradiction. Essentially, since $\chi(Z)=2g$ and our representation
  extends over $\pi_1(Z)$, Theorem~\ref{thm:upperbound}  implies that  $|\rho^{(2)}(Y,\Lambda)|$ is small,    while our assumptions on $|\rho_0(J^j_i)|$ together with Step 3,  Proposition~\ref{prop:additivity}, and Proposition~\ref{prop:7-1-MP} will imply
  that $|\rho^{(2)}(Y,\Lambda)|$ is very large.
\end{enumerate}
We now prove the two propositions crucial to Steps 1 and 2, respectively.

\begin{prop}\label{prop:step1}
Let $N \in \mathbb{N}$ be arbitrary and $K= \#_{i=1}^N K_i$ be a knot with $g_4(K) \leq g$.
Then there exists a compact connected 4-manifold $V$ such that, letting $U$ denote the standard cobordism between $M_K$ and $Y:=\sqcup_{i=1}^N M_{K_i}$ from Section~\ref{section:a-standard-cobordism}, $Z:= U \cup_{M_K} V$ satisfies:
\begin{enumerate}[(i)]
\item $\partial Z= Y$,
\item $H_2(Z)=\Z^{2g+N-1}$, and
\item   $\chi(Z)= 2g$.
\end{enumerate}
\end{prop}

\begin{proof}
Let $F'$ be a  locally flat surface embedded in $D^4$ with $\partial F'=K$ and $g(F')=g$.
Following \cite[Proposition~5.1]{Cha:2006-1}, we construct a topological
4-manifold $V$ with boundary $\partial V=M_K$, $H_1(M_K) \to H_1(V)$ an
isomorphism, and $H_2(V)\cong \Z^{2g}$, as follows. Let $X=X_0(K)$ denote the
0-trace of $K$, the 4-manifold obtained from $D^4$ by attaching a 0-framed
2-handle along a neighborhood of $K$. Let $F$ be the closed surface in $X$
obtained by taking the union of $F'$ with a core of the 2-handle.  Note that
since $F$ is locally flat it has a normal bundle by
\cite[Section~9.3]{Freedman-Quinn:1990-1}. Observe that $F \cdot F=0$, and so
$\nu(F) \cong F \times D^2$. Now, let $V= \left(X \ssm \nu(F)\right)
\cup_{\partial \nu(F)} H \times S^1$, where $H$ is any handlebody with $\partial
H= F$. A Mayer-Vietoris argument shows that $H_1(V) \cong \Z$, with generator a
meridian to $F$, and that $H_2(V) \cong \Z^{2g}$. Note that by Poincar{\'e}
duality and universal coefficients, we have
\[
  H_3(V) \cong H^1(V, M_K) \cong \Hom(H_1(V,M_K),\Z) \cong \Hom(0,\Z) \cong 0.
\]
So in particular the Euler characteristic of $V$ is $\chi(V)= 1-1+2g-0+0=2g$.

Let $U$ be the standard cobordism between $M_K$ and $\bigsqcup_{i=1}^N M_{K_i}$
discussed rather extensively in Section~\ref{section:a-standard-cobordism}. Now let $Z= V \cup_{M_K} U$, as
illustrated schematically in Figure~\ref{fig:diagramzz}.

\begin{figure}[H]
  \includegraphics[height=1.5in]{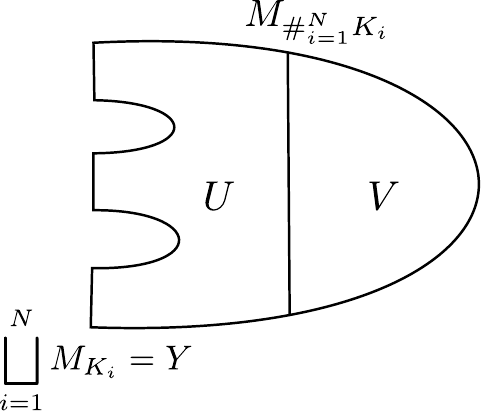}
  \caption{A schematic diagram of $Z= U \cup_{M_K} V$.}
  \label{fig:diagramzz}
\end{figure}

Note that $H_2(Z) \cong \Z^{2g} \oplus \Z^{N-1}$, and the inclusion induced map
$H_1(M_K) \to H_1(Z)$ is an isomorphism, as are each of the maps $H_1(M_{K_i})
\to H_1(Y) \to H_1(Z)$ for $i=1, \dots, N$. Also, $H_3(Z) \cong \Z^{N-1}$ and
$H_4(Z)=0$. So the Euler characteristic of $Z$ is
\[
  \chi(Z)= 1-1 + (2g + N-1)- (N-1))= 2g. \qedhere
\]
\end{proof}

\begin{prop}\label{prop:step2}
Let $R$ be a ribbon knot  and $p$ be a prime, and let $m_R$ denote the generating rank of  the $p$-primary part of $H_1(\Sigma_2(R))$.
 Fix $N \in \mathbb{N}$, and for each $i=1, \dots, N$ let $K_i$ be a knot obtained by infection along an unlink $\{\eta^j\}_{j=1}^r$ in the complement of $R$ such that  each $\eta^j$ represents an element of $\pi_1(M_R)^{(1)}$.
 Let $K= \#_{i=1}^N K_i$, and suppose that $M_K$ bounds a compact connected 4-manifold $V$ such that $H_1(M_K) \to H_1(V)$ is an isomorphism.

Then there exist $\chi_i \colon H_1(\Sigma_2(R)) \to \Z_p$, for $i=1, \dots, N$ such that:
 \begin{enumerate}[(a)]
 \item at least $\frac{m_RN-2 \chi(V)}{2}$ of the $\chi_i$ are nonzero, and
 \item  for some $a>0$, there exists a map $\pi_1(V) \to \Z \ltimes \Z_{p^a}$ such that the composition $\pi_1(M_K) \to \pi_1(V)\to \Z \ltimes \Z_{p^a}$ is given by the post-composition of $f_{\oplus_{i=1}^N \chi^{K_i}_i}$
with the inclusion $\Z \ltimes \Z_p \hookrightarrow \Z \ltimes \Z_{p^a}$.
 \end{enumerate}
  \end{prop}

\begin{proof}
For convenience, let $n=\frac{m_R N-2 \chi(V)}{2}$.
There is a canonical identification $H_1(M_K^2) \cong \Z \oplus H_1(\Sig_2(K))$,
and so given any $(\chi_i^{K_i})_{i=1}^N$ we obtain not just a map $\chi\colon
H_1(\Sig_2(K)) \to \Z_p$ but also  a map $\overline{\chi}\colon H_1(M_K^2) \to \Z_p$
by sending the $\Z$ coordinate to zero. Since the inclusion $H_1(M_K) \to
H_1(V)$ is an isomorphism, it therefore suffices to show that there are
homomorphisms $(\chi^R_i)_{i=1}^N\colon H_1(\Sig_2(R)) \to \Z_p$, at least $n$ of which are
nonzero, such that the map
\[
  \overline{\chi}:=\overline{(\chi^{K_i}_i)_{i=1}^N}\colon H_1(M_K^2) \to \Z_p
\]
extends over $H_1(V^2)$, perhaps after expanding its codomain to $\Z_{p^a}$ for
some $a>0$. Note that $\overline{\chi}$ extends over $H_1(V^2)$ up to enlarging its
codomain if and only if $\overline{\chi}$ vanishes on
\[
  H:=\ker(H_1(M_K^2) \xrightarrow{i_1} H_1(V^2)).
\]

The group of characters $TH_1(M_K^2) \to \Z_p$ is isomorphic to $H_1(\Sig_2(K),
\Z_p)$, which is in turn congruent to $(\Z_p^{m_R})^N$, where we recall that $m_R$
denotes the generating rank of the $p$-primary part of $H_1(\Sig_2(R))$. The
subgroup of characters vanishing on $H$ is in bijective correspondence with
$(TH_1(M_K^2)/H) \otimes \Z_p$.

Note that $M_K$ is a homology $S^1 \times S^2$ with
$\grr (TH_1(M^2) \otimes \Z_p)=m_RN$ Therefore, by Proposition~\ref{prop:homology}, the $p$-primary part of $TH_1(M_K^2)/H$ has generating rank at least~$\frac{m_RN- 2 \chi(V)}{2}=n$.  Therefore $TH_1(M_K^2)/H$ has a subgroup
isomorphic to~$\Z_p^n$. Our desired result now follows from a linear algebra argument (see the proof of \cite[Theorem~6.1]{KimLivingston:2005}): every subgroup of $\Z_p^{M}$ isomorphic to $\Z_p^{\ell}$ ($0 \leq \ell \leq M$) has an element at least $\ell$ of whose coordinates are nonzero.
\end{proof}


%

Now we  prove Theorem~\ref{thm:general}.

\begin{proof}[Proof of Theorem~\ref{thm:general}]
Suppose for the sake of contradiction that there is some locally flat surface
$F'$ embedded in $D^4$ with $\partial F'=K$ and $g(F')=g$.
Let $U, V,$ and $Z$ be as in Proposition~\ref{prop:step1}. Note that as discussed in Section~\ref{section:a-standard-cobordism} we have a standard choice of basepoints and paths
inducing inclusion maps; for the rest of the proof, these choices will remain
fixed though not explicitly discussed.

We pause to establish notation. For a knot $J$ in $S^3$, we denote its exterior
by~$E_J$.  For a manifold $X$ with $H_1(X) \cong \Z$, we denote its canonical
double cover by $X^2$. The choice of a meridian $\mu_J$ determines a splitting
$\pi_1(M_J) \cong \Z \ltimes \mathcal{A}(J)$, where $\mathcal{A}(J)$ denotes the
Alexander module of $J$. Note that $H_1(\Sig_2(J))$ is naturally identified with
$\mathcal{A}(J)/ \langle t+1 \rangle$, and so a map $\chi\colon H_1(\Sig_2(J))
\to \Z_p$ induces a map
\[
  f_\chi\colon \pi_1(M_J) \xrightarrow{\cong} \Z \ltimes \mathcal{A}(J)
  \to \Z \ltimes H_1(\Sig_2(J)) \xrightarrow{Id \ltimes \chi} \Z \ltimes \Z_p.
\]

Note that in the setting of Proposition~\ref{prop:step1}, since $H_1(M_K) \to H_1(Z) \cong \Z$ is an isomorphism,
$Z$ also has a canonical double cover $Z^2$. It is easy to check that $Z^2= V^2
\cup_{M_K^2} U^2$  and that $\partial Z^2= \bigsqcup_{i=1}^{N} M^2_{K_i}$.

For each $i=1, \dots, N$, we have a canonical, linking form--preserving
identification of $H_1(\Sig_2(K_i))$ with $H_1(\Sigma_2(R))$ coming from the
degree one maps $E_{J_i^j} \to E_{\text{unknot}}$. Given a map $\chi^R\colon
H_1(\Sig_2(R)) \to \Z_p$ we will use $\chi^{K_i}$ to denote the corresponding
map from $H_1(\Sig_2(K_i)) \to \Z_p$, and vice versa.  We will also always
identify $H_1(\Sigma_2(K))$ with $\bigoplus_{i=1}^N H_1(\Sigma_2(K_i))$ in the
canonical, linking form--preserving way.


Define $n := \frac{m_RN-4g}{2}$. (Note that with $\chi(V)=2g$ this agrees with the definition of $n$ used above.) We wish to show that there exist $\chi^R_i\colon
H_1(\Sigma_2(R)) \to \Z_p$, for $i=1, \dots, N$, such that at least $n$ of the
$\chi^R_i$ are nonzero and for some $a>0$, there exists a map $\pi_1(Z) \to \Z
\ltimes \Z_{p^a}$ such that the composition $\pi_1(M_{K_i}) \xrightarrow{\iota
\circ \iota_i} \pi_1(Z)\to \Z \ltimes \Z_{p^a}$ is given by the postcomposition of $f_{\chi^{K_i}_i}$
with the inclusion $\Z \ltimes \Z_p \hookrightarrow \Z \ltimes \Z_{p^a}$.
Henceforth, we will implicitly take the usual inclusion of
$\Z_p$ in $\Z_{p^a}$ without further comment.

We will accomplish this in a somewhat indirect fashion, by focusing on
constructing an appropriate map on $\pi_1(M_K)$ which extends over $\pi_1(U)$
and $\pi_1(V)$ separately. By Proposition~\ref{prop:standardu}, given any choice
of $\chi^R_1, \dots, \chi^R_N\colon H_1(\Sigma_2(R)) \to \Z_p$, the map
$f_{(\chi^{K_i}_i)_{i=1}^N}\colon \pi_1(M_K) \to \Z \ltimes \Z_p$ extends
uniquely to a map $F\colon \pi_1(U) \to \Z \ltimes \Z_p$ such that when we
consider the composition
\[
  F \circ\iota \circ \iota_i\colon \pi_1(M_{K_i}) \to \Z \ltimes \Z_p,
\]
we have $F \circ \iota \circ\iota_i= f_{\chi^{K_i}_i}$.
By applying Proposition~\ref{prop:step2} to our $K$ and $V$ and extending over $U$ as discussed above, we obtain $\chi= (\chi^R_1, \dots, \chi^R_{N})$ with at least $\frac{m_R N-4g}{2}=:n$ of the $\chi_i$ nonzero together with a map
$F\colon \pi_1(Z) \to \Z \ltimes \Z_{p^a}$ such that the composition
\[
  \pi_1(M_{K_i}) \xrightarrow{\iota \circ \iota_i} \pi_1(Z) \to \Z \ltimes \Z_{p^a}
\]
is given by $f_{\chi^{K_i}_i}$.\\
As described in Section~\ref{section:metabelian-homology}, we have a fixed map $\theta \colon \Z
\ltimes \pi_1(\Sigma_2(K)) \to \GL_2(\Q(\xi_{p^a})[t^{\pm1}])$. By
post-composing $F$ and each $f_{\chi^{K_i}_i}$ with this map, we obtain
\begin{align*}
  \Phi= \theta \circ F\colon& \pi_1(Z) \to  \GL_2(\Q(\xi_{p^a})[t^{\pm1}]),
  \\
  \phi_i= \theta \circ f_{\chi_i^{K_i}}\colon& \pi_1(M_{K_i})
  \to  \GL_2(\Q(\xi_{p^a})[t^{\pm1}]).
\end{align*}
We let
\[
  \phi= \coprod_{i=1}^N \phi_i\colon \coprod_{i=1}^N \pi_1(M_{K_i}) \to
  \GL_2(\Q(\xi_{p^a})[t^{\pm1}]).
\]

For convenience, let $\F= \Q(\xi_{p^a})$, $S= \F[t^{\pm1}]$,
$Q= \F(t)$, and  $S/ p$ be shorthand for $S/ \langle p(t) \rangle$ for any
polynomial $p(t) \in S$. Since our infection curves $\alpha_i$ live in the
second derived subgroup of $M_{R_0}$, the degree one maps $f_i\colon E_{J_i^j}
\to E_{\text{unknot}}$ give us an identification
\[
  f_*\colon H_1^{\phi}(Y, S) =
  \bigoplus_{i=1}^{N}  H_1^{\phi_i}(M_{K_i}, S)
  \xrightarrow{\cong} \bigoplus_{i=1}^{N}  H_1^{\theta \circ f_{\chi_i^R},S}(M_{R})
\]
where the maps $\chi^R_i\colon H_1(\Sigma_2(R)) \to \Z_p \hookrightarrow
\Z_{p^a}$ are as above.
We now work towards proving the following claim.

\begin{claim}\label{claim:not-contained}
 \[ H:= \bigoplus_{\{i \mid \chi_i^{R} \neq 0\}} H_1^{\phi_i}(M_{K_i})
\text{ is not contained in }\ker( H_1^{\phi}(Y; \F[t^{\pm1}]) \to H_1^{\Phi}(Z; \F[t^{\pm1}])).\]
\end{claim}
\begin{proof}[Proof of Claim~\ref{claim:not-contained}]
First, note that $H$ has generating rank at least $\lceil k/d_R \rceil $, since for some nontrivial $\chi_0 \colon H_1(\Sigma_2(R)) \to \Z_p$ there is a submodule of $H$ isomorphic to $\left( H_1^{\theta \circ f_{\chi_0}}(M_R) \right)^{\lceil k/{d_R}\rceil}$.
 Note that if $\chi_i^R=0$ then
\[
  H_1^{\theta \circ f_{\chi_i^R}}(M_{R} )
  \cong\mathcal{A}_{\Q} (R) \otimes_{\Q[t^{\pm1}]} \F[t^{\pm1}].
\]
We therefore have that
\[
  H_1^{\phi}(Y, S) \cong (\mathcal{A}_{\Q} (R) \otimes_{\Q[t^{\pm1}]} \F[t^{\pm1}])^{N-k}
  \oplus H
\]
where $k \geq \frac{m_R N-4g}{2}$ is the number of nonzero $\chi_i^R$.
%

We now compute the rank of $H_2^{\Phi}(Z; Q)$. We can immediately see that
$H_0^{\Phi}(Z; Q)) \cong H_0^{\phi}(Y; Q) \cong 0$, since $H_0^{\Phi}(Z)$ and
$H_0^{\phi}(Y)$ are annihilated by $t-1$.  Note that for each $i=1, \dots, N$
the inclusion map $Y_i \to Z$ induces an isomorphism on $H_0(-; \Z)$ and
$H_1(-;\Z)$. By the proof of \cite[Proposition~4.1]{Friedl-Powell:2010-1}, modified to use only a partial chain contraction for $C_*^{\Phi}(Z,Y_i;Q)$ in degrees $0,1$, as in \cite[Proposition~2.10]{Cochran-Orr-Teichner:1999-1},  this implies that the map
$H_1^{\phi_i}(Y_i; Q) \to H_1^{\Phi}(Z; Q)$ is onto. We have already observed
that $H_1^{\phi_i}(Y_i)$ is a torsion $S$-module and so $H_1^{\phi_i}(Y_i;
Q)=0$; it follows that $H_1^{\Phi}(Z; Q)=0$ as well. Consideration of the long exact sequence of the
pair $(Z,Y)$ then allows us to conclude that $H_1^{\Phi}(Z, Y;  Q)=0$.
By Poincar{\'e}-Lefschetz
duality, universal coefficients, and the long exact sequence of $(Z, Y)$ with
$Q$-coefficients we have that
\[
  H_3^{\Phi}(Z;Q) \cong H^1_{\Phi}(Z,Y; Q) \cong
  \Hom(H_1^{\Phi}(Z, Y;  Q), Q) \cong 0.
\]
Finally, since $Z$ is a topological 4-manifold and hence homotopy equivalent to a finite
CW complex with cells of dimension at most 3 (see the proof of
Theorem~\ref{thm:upperbound} for references for this fact), we have that
$H_j^{\Phi}(Z; \F(t))=0$ for all $j \geq 4$. Re-computing $\chi(Z)$ with
$Q$-coefficients, we obtain
\[
  2g= \chi(Z)= 0 - 0 + \dim_{Q}  H_2^{\Phi}(Z; Q) -0+0
  =\dim_{Q}  H_2^{\Phi}(Z; Q).
\]

We now return to working with $S = \mathbb{F}[t^{\pm 1}]$-coefficients and
consider the long exact sequence of Proposition~\ref{prop:les-pairs}
\[
  \dots \to H_2^{\phi}(Y) \xrightarrow{i_2} H_2^\Phi(Z)
  \xrightarrow{j_2} H_2^\Phi(Z,Y) \xrightarrow{\partial} H_1^{\phi}(Y)
  \xrightarrow{i_1} H_1^\Phi(Z) \xrightarrow{j_1} H_1^\Phi(Z,Y) \to \dots.
\]
Suppose now for a contradiction that
$H \leq \ker(i_1)$. Since
\begin{align*}
 \ker(i_1) =
  \Imm(\partial) &\cong H_2^\Phi(Z,Y)/ \ker(\partial) =
  H_2^\Phi(Z,Y)/ \Imm(j_2) \cong \left( S^{2g} \oplus TH_2^\Phi(Z,Y) \right)/ \Imm(j_2),
\end{align*}
it follows that $\left( S^{2g} \oplus TH_2^\Phi(Z,Y) \right)/ \Imm(j_2)$ has a submodule $H'$ isomorphic to $H$.

By applying Lemma~\ref{lem:intersecttorsion} with $A= H_2^{\Phi}(Z,Y)$, $B= \Imm(j_2)$, and $C= H'$ we obtain that  $TH_2^{\Phi}(Z,Y) / (\Imm(j_2) \cap TH_2^{\Phi}(Z,Y) )$ contains a submodule $H''$ of generating rank at least $\lceil k/{d_R}\rceil-2g$ and of order which divides the order of $H'$ and so is relatively prime to $\Delta_R$.
Since  $\Imm(j_2|_T) \subseteq \Imm(j_2) \cap TH_2^{\Phi}(Z,Y)$, it follows immediately that $\coker(j_2|_T)=TH_2^\Phi(Z,Y)/ \Imm(j_2|_T)$ contains a submodule of generating rank at least $\lceil k/{d_R}\rceil-2g$ and of order relatively prime to $\Delta_R$.

As argued above, we have that $H_1^{\Phi}(Z,Y;Q) =0$, i.e. that  $H_1^{\Phi}(Z,Y)$ is torsion, and so we can apply Lemma~\ref{lemma:moregeneralhomology} to conclude that
\begin{align*}
\coker(j_2|_T)) \cong \ker(j_1|_T)= \ker(j_1)= \Imm(i_1).
\end{align*}
Therefore   $\Imm(i_1)$ has a submodule of generating rank at least $\lceil k/{d_R}\rceil-2g$ and  of order that is relatively prime to $\Delta_R$.
Since $k \geq n= \frac{m_R N-4g}{2}$ and $N \geq \frac{4g(d_R+1)+2}{m_R}$, we obtain that $\lceil k/d_R\rceil -2g >0$
and so there is a submodule of
$\Imm(i_1)$ isomorphic to $S/s$ for some nontrivial polynomial $s$ relatively
prime to $\Delta_R$. This is our desired contradiction, since
$H  \leq \ker(i_1)$
also implies that $\Imm(i_1)$ is a quotient of $\left(\mathcal{A}_{\Q} (R) \otimes_{\Q[t^{\pm1}]} \F[t^{\pm1}]\right)^{N-k}$, which has order $\Delta_R^{N-k}$ and therefore cannot contain a submodule isomorphic
to~$S/s$.
This completes the proof of the claim. \end{proof}

\begin{claim}\label{prop:step-3}
  For some $1 \leq i \leq N$ and $1 \leq j \leq r$,  the element $ [1,0] \otimes [\lambda(\eta_i^j)] $ does not map to $0$ in $H_1^{\Phi}(Z)$.
\end{claim}

\begin{proof}[Proof of Claim~\ref{prop:step-3}]
Observe that since the longitude $\lambda(\eta_i^j)$ of $\eta_i^j$ is in the
second derived subgroup of $\pi_1(M_{R})$  it must lift to a curve $l_i^j$ in
the cover $\widetilde{M_{R}}$ of $M_{R}$ determined by $\phi_i$. (In fact, it
lifts to  $\Z \ltimes \Z_{p^a}$ copies -- pick one.)
Since whenever $\chi_i^R \neq 0$ we have that the collection $\{[1,0] \otimes [l_i^j]\}_{j=1}^r$ generates $H_1^{\phi_i}(M_R)$, our argument that $H \not \leq \ker(i_1)$ in fact implies that
for at least one $i$ and $j$ with $1 \leq i \leq N$ and $1\leq j \leq r$, we have $i_1([1,0] \otimes
[l_i^j]) \neq 0$ in $H_1^{\Phi}(Z).$  This completes the proof of Claim~\ref{prop:step-3} and of Step~3.
\end{proof}

We are now ready to complete the proof of Theorem~\ref{thm:general}, as described in Step 4, by constructing a new representation of $\pi_1(Y)$  and bounding $\rho^{(2)}(Y, \psi)$ in two different ways to derive a contradiction.
Let
\[\psi \colon\pi_1(Y) \to \pi_1(Z) \to \Lambda:= \pi_1(Z)/ \pi_1(Z)^{(3)}_{(\Q, \Z_{p^a}, \Q)}\]
 be the map induced by inclusion. Since
$\Lambda$ is amenable and in $D(\Z_p)$~\cite[Lemma~4.3]{Cha:2014-1} and
$\psi$ evidently extends over $\pi_1(Z)$,  Theorem~\ref{thm:upperbound} and the
fact from Step 1 that $H_2(Z) \cong \Z^{2g+N-1}$ tells us that
\begin{align}\label{equation:upper-bound}
  |\rho^{(2)}(Y, \psi)| \leq 2 \dim_{\Z_p}H_2(Z, \Z_p) = 2(2g+ N-1).
\end{align}

Let $(i_0, j_0)$ be the maximal tuple (with respect to the lexicographic ordering) such that $[1,0] \otimes [l_i^j]$ does not map to $0$ in $H_1^\Phi(Z)$.
 Proposition~\ref{prop:7-1-MP} implies that $\lambda(\eta_{i_0}^{j_0}) \not \in
\pi_1(Z)^{(3)}_{(\Q, \Z_{p^a}, \Q)}$.
Moreover, Proposition~\ref{prop:ouradditivity} tells us that, letting $\delta_i^j(\psi)=\begin{cases} 1, & \psi(\lambda(\eta_i^j)) \neq 0 \\ 0, &\psi(\lambda(\eta_i^j))=0 \end{cases}$, we have
\begin{align} 
  \rho^{(2)}(Y, \psi)&=\sum_{i=1}^{N}\rho^{(2)}(M_{R_{\alpha}(J_i)}, \psi|_{M_{R_{\alpha}(J_i)}}) = \sum_{i=1}^{N}\Big(\rho^{(2)}(M_{R}, \psi^0_{i})+\sum_{j=1}^r \delta_i^j(\psi)
  \rho_0(J_i^j) \Big).
  \label{eqn:siglargept2}
\end{align}
 Since $|\rho^{(2)}(M_{R}, \psi^0_{i})| \leq C_R$ for all $i$, the tuple $(i_0, j_0)$ is maximal such that $ \delta_i^j(\psi) \neq 0$, and  $J_{i_0}^{j_0}$ satisfies
\[ |\rho_0(J^{j_0}_{i_0})|> 2(2g+N-1) + N C_R
  + \sum_{k=1}^{i_0-1} \sum_{\ell=1}^r |\rho_0(J^\ell_k)|
  + \sum_{\ell=1}^{j_0-1} |\rho_0(J^\ell_i)|,
\]
Equation~\ref{eqn:siglargept2} gives the desired contradiction with Equation~\ref{equation:upper-bound}, which completes the proof of Theorem~\ref{thm:general}.
\end{proof}

\section{Height four gropes}\label{section:height-four-gropes}

In Proposition~\ref{prop:gropebounding} below, we will show the following: the knot $K$ in
Section~\ref{subsection:example2bipolar} bounds a framed grope of height 4
embedded in~$D^4$.  For the reader's convenience, we begin by recalling the
definition of a (capped) grope, a certain type of 2-complex.

\begin{defn}
  [Grope of height $h$~{\cite{Freedman-Quinn:1990-1,
  Cochran-Orr-Teichner:1999-1}}] A \emph{capped surface}, or a \emph{capped
  grope of height 1,} is an oriented surface of genus $g>0$ with nonempty
  connected boundary, together with discs attached along the $2g$ curves of a
  standard symplectic basis for the surface. The discs are called \emph{caps}.
  If $G$ is a capped grope of height $h-1$, then a 2-complex obtained by
  replacing each cap of $G$ with a capped surface is called a \emph{capped grope
  of height $h$}.  A \emph{grope of height $h$} is obtained by removing caps
  from a capped grope of height $h$. It is also called the \emph{body} of the
  capped grope. The initial surface that the inductive construction starts with
  is called the \emph{base surface}, and the \emph{boundary of a grope},
  $\partial G$, is the boundary of its base surface.
\end{defn}

A (capped) grope defined above is often called \emph{disc-like}.  An
\emph{annulus-like} (capped) grope is defined in the same way, starting from a
base surface with two boundary components.

\begin{remark}
  It is not a priori obvious that a 2-complex $G$ known to be a grope has a
  well-defined height, but it is true. For the reader's convenience, we give a
  quick argument. Let $\tau \subset G$ be the singular set of the grope union
  its boundary, i.e.\ the 1-complex consisting of the points where~$G$ is not
  locally homeomorphic to an open disc. Then $G \ssm \tau$ consists of a
  collection of open surfaces, many of which are planar. Removing the subset of
  $G$ corresponding to the non-planar surfaces (the interior of the `top stage'
  of $G$) gives a new grope with a strictly smaller singular set; we can then
  repeat the above procedure. In this perspective, the height of a grope is
  exactly  the number of such steps needed to reduce the grope to a circle. We
  leave to the reader the analogous argument that the height of a capped grope
  is well-defined, as well as the intrinsic definition of the $i$th stage of a
  grope, $1 \leq i \leq h$.
\end{remark}

A (capped) grope admits a standard embedding in the upper half 3-space
$\mathbb{R}^3_+=\{z\ge 0\}$ which takes the boundary to $\mathbb{R}^2$. Compose
it with $\mathbb{R}^3_+\hookrightarrow \mathbb{R}^4_+$, take a regular
neighborhood in $\mathbb{R}^4_+$, and possibly perform finitely many plumbings.
An embedding of the result in a 4-manifold is called an \emph{immersed framed
(capped) grope}.  If no plumbing is performed, then we say that it is
\emph{embedded}.  Often we will regard an immersed/embedded (capped) grope as a
2-complex, but it is always assumed to be framed in this sense.  In addition, we
assume that each intersection in an immersed capped grope is always between a
cap and a surface in the body, following the convention
of~\cite{Cha-Kim:2016-1}.  Note that in a simply connected 4-manifold, an
embedded grope without caps can be promoted to an immersed capped grope.

Returning to our case, recall that the knot $K$ in
Section~\ref{subsection:example2bipolar} is the connected sum of satellite knots.
We will use the following terminology and results
from~\cite{Cha:2012-1,Cha-Kim:2016-1}, which also consider link versions.

\begin{defn}[Satellite capped
  grope~{\cite[Definition~4.2]{Cha:2012-1}},
  {\cite[Definition~4.2]{Cha-Kim:2016-1}}]
  \label{definition:satellite-capped-grope}
  Suppose $K$ is a knot in $S^3$ and $\alpha$ is an unknotted circle in $S^3$
  disjoint from~$K$.  Let $E_\alpha$ be the exterior of $\alpha$, and let
  $\lambda_\alpha$ be a zero linking longitude on~$\partial E_\alpha$.  A
  \emph{satellite capped grope} for $(K,\alpha)$ is a disc-like capped grope $G$
  immersed in $E_\alpha\times I$ such that the boundary of $G$ is
  $\lambda_\alpha\times 0$, the body of $G$ is disjoint from $K\times I$, and
  the caps are transverse to $K\times I$.
\end{defn}

\begin{defn}[Capped grope concordance~{\cite[Definition~4.3]{Cha-Kim:2016-1}}]
  \label{definintion:grope-concordance}
  A \emph{capped grope concordance} between two knots $J$ and $J'$ is an
  annulus-like capped grope immersed in $S^3\times I$ such that the base surface
  is bounded by $J\times 0 \cup -J'\times 1$.
\end{defn}

\begin{prop}[{\cite[Section~4.1]{Cha-Kim:2016-1}}]
  \label{proposition:product-satellite-grope}
  Suppose that there is a satellite capped grope of height $h$ for $(K,\alpha)$
  and a capped grope concordance of height $\ell$ between two knots $J$
  and~$J'$.  Then there is a capped grope concordance of height $h+\ell$ between
  the satellite knots $K_{\alpha}(J)$ and~$K_{\alpha}(J')$.
\end{prop}

The height $h+\ell$ capped grope concordance in
Proposition~\ref{proposition:product-satellite-grope} is obtained by a
``product'' construction described in \cite[Definition~4.4]{Cha-Kim:2016-1}.
The last ingredient we need is the following result
from~\cite{Cochran-Orr-Teichner:1999-1}.

\begin{prop}[{\cite[Remark~8.14]{Cochran-Orr-Teichner:1999-1}}]
  \label{proposition:height-two-grope}
  A knot in $S^3$ with trivial Arf invariant bounds a capped grope of height two
  immersed in~$D^4$.
\end{prop}

We can now prove the following.
\begin{prop} \label{prop:gropebounding}
Let $K= \#_{i=1}^n R_{\alpha^+,\alpha^-}(J_i^+, J_i^-)$ be a connected sum of satellite knots, where $(R,\alpha^+,\alpha^-)$ is as in the right of Figure~\ref{fig:11n742} and  $\{J_i^+, J_i^{-}\}_{i=1}^n$ a collection of knots with vanishing Arf invariant.
Then $K$ bounds an embedded grope of height 4 in~$D^4$.
\end{prop}

\begin{proof}
First, note that it suffices to show that each $ R_{\alpha^+,\alpha^-}(J_i^+, J_i^-)$ bounds
an embedded grope of height 4, since we can then take the boundary connected sum of
such gropes to obtain one with boundary $K$. We therefore show that under the hypothesis that $\Arf(J^+)= \Arf(J^-)=0$ the knot $R_{\alpha^+,\alpha^-}(J^+, J^-)$ bounds a grope of height 4.

Observe that the curve $\alpha^-$ in
Figure~\ref{fig:11n742} bounds a disjoint capped grope of height two embedded
in~$S^3$, where the body surfaces are disjoint from the knot $R$ but the caps
are allowed to intersect~$R$.  This is a geometric analogue of the commutator
relation $\alpha^-=[\beta_1, \beta_2]$ where the curves $\beta_1$ and $\beta_2$
shown in the left of Figure~\ref{fig:11n742} are again commutators in the
fundamental group.

\begin{figure}[ht]
  \includegraphics[height=7cm]{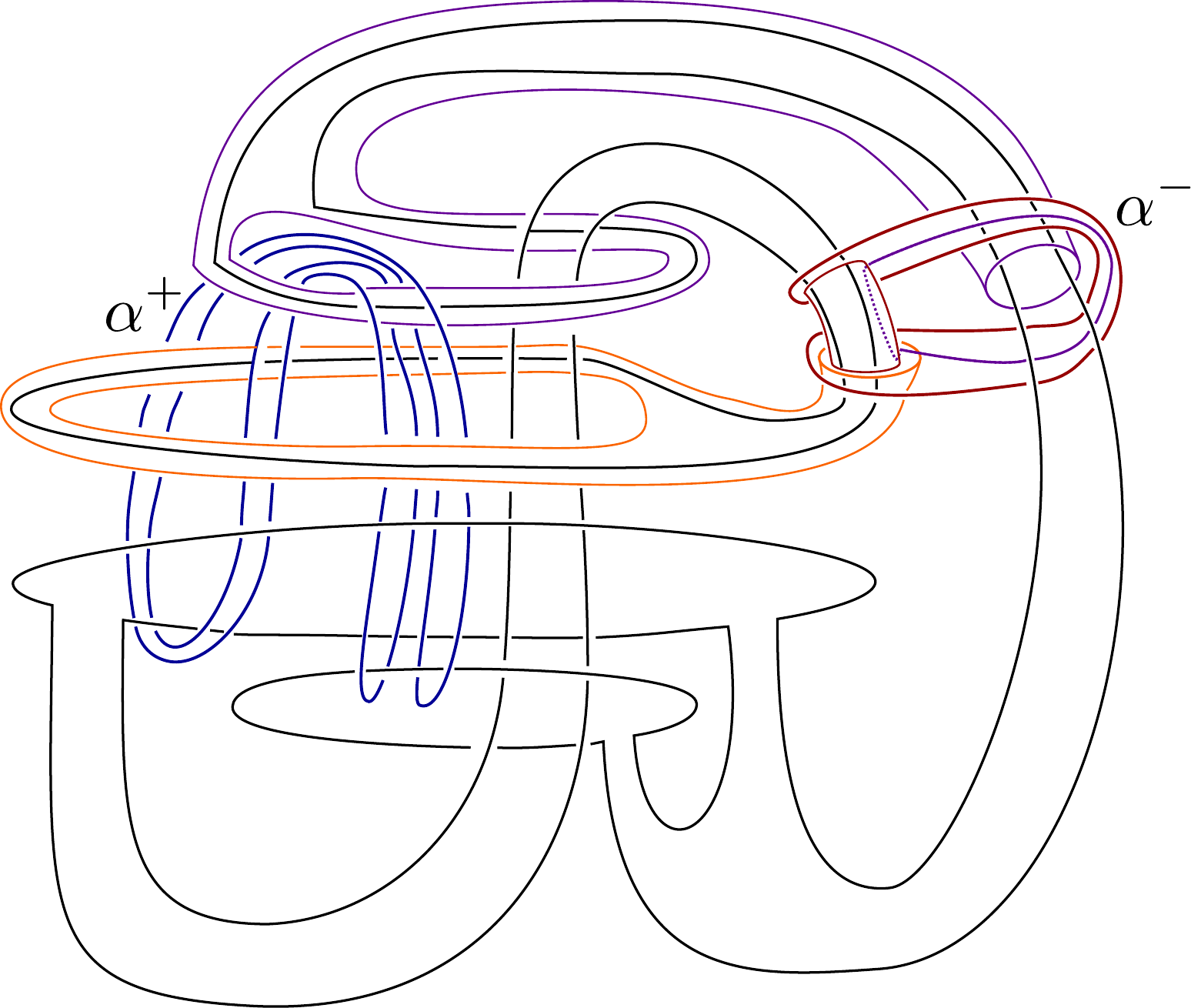}
  \caption{An embedded height 2 grope with boundary $\alpha^-$ in $S^3 \ssm (R \sqcup \alpha^+)$.}
  \label{fig:height2grope}
\end{figure}

Indeed, in the planar diagram in the right of Figure~\ref{fig:11n742}, the
bounded region enclosed by $\alpha^-$ is the projection of an obviously seen
embedded disc which intersects $R$ in four points, and by tubing on this disc,
one obtains a genus one surface, shown in red in Figure~\ref{fig:height2grope},
which is disjoint from~$R$.  This surface is the base surface of the promised
height two grope bounded by~$\alpha^-$.  The curves $\beta_1$ and $\beta_2$ are
parallel to standard basis curves of the base surface, and they bound disjoint
genus one surfaces obtained by tubing the obviously seen discs along the
knot~$R$, as illustrated in Figure~\ref{fig:height2grope}.  Attach them to the
base stage surface to obtain a height two grope.

Note that all the surfaces used above are disjoint from the other curve
$\alpha^+$, so by performing the satellite construction, we obtain a height two
grope in $S^3\ssm R_{\alpha_+}(J^+)$ bounded by~$\alpha^-$.  Identify $S^3$ with
$S^3\times 0 \subset S^3\times I$, push the interior of the grope into the
interior of $S^3\times I$, and add caps using the simple connectedness of
$S^3\times I$ as noted above.  Apply general position to make the caps
transverse to $R_{\alpha_+}(J^+)\times I$, to obtain a satellite capped grope
for~$(R_{\alpha_+}(J^+),\alpha^-)$.

Since the knot $J^-$ has trivial Arf invariant, $J^-$ bounds a capped
grope of height two in~$D^4$, by Proposition~\ref{proposition:height-two-grope}.
Remove, from $D^4$, a small open 4-ball which intersects the capped grope in an
unknotted 2-disc lying in the interior of the base surface, to obtain a capped
grope concordance of height two between $J^-$ and the trivial knot.  By
Proposition~\ref{proposition:product-satellite-grope} and the above paragraph,
 the satellite knot $R_{\alpha^+,\alpha^-}(J^+,J^-) =
(R_{\alpha^+}(J^+))_{\alpha^-}(J^-)$ is height 4 capped grope concordant to
the knot~$R_{\alpha^+}(J^+)$.  Forget the caps of this capped grope concordance,
and attach a slicing disc for the knot $R_{\alpha^+}(J^+)$, to obtain a grope of
height 4 bounded by~$R_{\alpha^+,\alpha^-}(J^+,J^-)$.
\end{proof}

\begin{remark}
  A similar argument shows the existence of a bounding grope of height 4 for the
  simpler example in Section~\ref{subsection:example2solvable}.  In this case,
  the height two surfaces constructed by ``tubing along the knot $R$'' in the
  3-space are not disjoint, but the intersection can be removed by pushing the
  surfaces into 4-space.   We omit the details.
\end{remark}

\bibliographystyle{amsalpha}
\def\MR#1{}
\bibliography{research}
\end{document}